\documentclass[
final
, nomarks
]{dmtcs-episciences}

\usepackage[utf8]{inputenc}
\usepackage{subfigure}

\usepackage[round]{natbib}

\usepackage{environ}
\usepackage{tabularx}
\usepackage{amsthm}
\usepackage{thm-restate}
\usepackage{amsmath}

\newcommand{\hide}[1]{}

\newcommand{\R}{\mathbf{r}}
\newcommand{\B}{\mathbf{b}}

\newcommand{\LL}{\mathcal{L}}
\newcommand{\HH}{\mathcal{H}}
\newcommand{\KK}{\mathcal{K}}
\newcommand{\OO}{\mathcal{O}}

\newcommand{\VB}{\textsc{VB}}

\newcommand{\E}{\mathbb{E}}
\newcommand{\depth}{\textsc{depth}}
\newcommand{\Spl}{\textrm{Spl}}

\newcommand{\cp}{\ensuremath{\textrm{cap}}}

\newtheorem{theorem}{Theorem}

\newtheorem{lemma}[theorem]{Lemma}

\newtheorem{definition}[theorem]{Definition}
\newtheorem{observation}[theorem]{Observation}
\newtheorem{proposition}[theorem]{Proposition}

\makeatletter
\newcommand{\problemtitle}[1]{\gdef\@problemtitle{#1}}
\newcommand{\probleminput}[1]{\gdef\@probleminput{#1}}
\newcommand{\problemquestion}[1]{\gdef\@problemquestion{#1}}
\NewEnviron{problem}{
  \problemtitle{}\probleminput{}\problemquestion{}
  \BODY
  \par\addvspace{.5\baselineskip}
  \noindent
  \begin{tabularx}{\textwidth}{@{\hspace{\parindent}} l X c}
    \multicolumn{2}{@{\hspace{\parindent}}l}{\@problemtitle} \\
    \textbf{Input:} & \@probleminput \\
    \textbf{Question:} & \@problemquestion
  \end{tabularx}
  \par\addvspace{.5\baselineskip}
}
\makeatother

\newtheoremstyle{case}{}{}{}{}{}{:}{ }{}
\theoremstyle{case}

\author[R. Raman and K. Singh]{Rajiv Raman
  \and Karamjeet Singh
  }
\title{On Supports for Graphs of Bounded Genus\thanks{A preliminary version of these results appeared in EUROCOMB'23~\cite{raman2023hypergraph}.}}
\affiliation{
  Indraprastha Institute of Information Technology, Delhi, India.}
\begin{document}



\publicationdata{vol. 28:2}{2026}{28}{10.46298/dmtcs.16086}{2025-07-22; 2025-07-22; 2026-03-13}{2026-05-02}

\maketitle
\renewcommand{\thefootnote}{\arabic{footnote}}
\begin{abstract}
 Let $(X,\mathcal{E})$ be a hypergraph. A support is a graph $Q$ on $X$ such that for each
$E\in\mathcal{E}$, the subgraph of $Q$ induced on the elements in $E$ is connected. 
We consider hypergraphs defined by connected subgraphs of a host graph.
For a graph $G=(V,E)$, let $\B(V)\subseteq V$ denote a set of \emph{terminals}.
Given a collection $\mathcal{H}$ of connected subgraphs of $G$, 
we define a hypergraph on $\B(V)$, where each $H\in\mathcal{H}$ defines
a hyperedge $V(H)\cap\B(V)$. Our goal is to construct a graph $Q$ on $\B(V)$
so that for each $H\in\mathcal{H}$, $V(H)\cap\B(V)$ induces a connected
subgraph of $Q$.

We also consider the problem of constructing a support for the \emph{dual hypergraph} - a hypergraph on $\mathcal{H}$ where each 
$v\in \B(V)$ defines
a hyperedge consisting of the subgraphs in $\mathcal{H}$ containing $v$.
In fact, we construct supports for a common generalization
of the primal and dual settings called the 
\emph{intersection hypergraph}.

As our main result, we show that if the host graph $G$ 
has genus $g$ and the subgraphs in $\mathcal{H}$ satisfy a 
condition of being \emph{cross-free}, then there exists a support
that also has genus at most $g$.

Our results are a generalization of the results of Raman and Ray (Rajiv Raman, Saurabh Ray:
Constructing Planar Support for Non-Piercing Regions. Discret. Comput. Geom. 64(3): 1098-1122 (2020)) and our techniques extend their results from
the planar setting to graphs on surfaces. In particular,
our techniques imply a unified analysis for packing and covering problems
for hypergraphs defined on surfaces of bounded genus. We also describe
applications of our results for hypergraph colorings.
\end{abstract}

\section{Introduction}\label{sec:intro}
A hypergraph $(X,\mathcal{E})$ consists of a set $X$ and a collection $\mathcal{E}$ of subsets of $X$ called \emph{hyperedges}.
Hypergraphs generalize graphs (where $\mathcal{E}$ consists only of 2-element subsets of $X$). 
We are interested in questions of representing hypergraphs by \emph{sparse} graphs that capture some of its structure. 
The notion of structure we want to preserve is that of connectedness of the hyperedges.

Consider a hypergraph $(X,\mathcal{E})$. A graph $Q$ on $X$ is a \emph{support} for the hypergraph if
for each $E\in\mathcal{E}$, the induced graph $Q[E]$ of $Q$ is connected. 
If we ignore the sparsity, we can trivially construct a support by taking a complete graph on $X$. The problems become
non-trivial if we impose sparsity. For example, if we require the support graph to be planar, then even
deciding if a hypergraph admits such a support is NP-hard due to a result of~\cite{Johnson1987Hypergraph}.

The notion of a support was introduced by~\cite{voloshina1984planarity} in the context of VLSI circuits (see~\cite{Feinberg1997} for an English version), and
has since been rediscovered by several communities. For example, the existence and construction of sparse supports arises naturally
in the problem of visualizing hypergraphs, see~\cite{bereg2015colored,bereg2011red,brandes2010blocks,brandes2012path,buchin2011planar,havet2022overlaying,hurtado2018colored},
and in the design of networks, see~\cite{chen2015polynomial,hosoda2012approximability,korach2003clustering,onus2011minimum}.

Since it is already NP-hard to decide if a hypergraph admits a support that is a planar graph, it is natural to consider restricted settings that ensure the existence of a sparse support.~\cite{RR18} considered hypergraphs defined by
\emph{non-piercing regions}\footnote{A slightly imprecise definition is the following: a family $\mathcal{R}$ of path connected regions in the plane or on an oriented surface such that for any two regions $R, R'\in\mathcal{R}$, the regions
$R\setminus R'$ and $R'\setminus R$ are also connected are called non-piercing regions. 
They include disks, pseudodisks, unit height-axis parallel rectangles, etc. See Def. \ref{def:npregions}.} in the plane (see Section~\ref{sec:regionsonsurfaces}
for a definition of these hypergraphs). The authors showed the existence
of planar \emph{primal, dual} and \emph{intersection} supports 
for these hypergraphs. 

In~\cite{RR18}, the authors were motivated by the analysis of algorithms
for Packing and Covering problems for hypergraphs defined by geometric
regions in the plane. Several such problems admit PTAS via algorithms from a basic local-search framework.
To show that the algorithms in this framework are a PTAS requires showing the existence
of a \emph{local-search graph}.  The existence of a support in many cases
directly implies the existence of such a local-search graph (though
in some cases, some extra work is required).

We introduce the notion of hypergraphs defined by connected subgraphs
of a host graph. This yields a cleaner combinatorial model to study
the geometric hypergraphs considered in~\cite{RR18}.
To see this, consider an arrangement of disks in the plane.
The \emph{dual arrangement graph} $G$ corresponding to the disks is a plane graph that has a vertex for each \emph{cell} of the arrangement, 
and two vertices in $G$ are adjacent if and only if the corresponding cells are separated by the boundary of a disk.
Each disk then defines a connected subgraph of $G$ - the subgraph induced on the cells contained in that disk. 
The subgraphs defined by any pair of disks are \emph{non-piercing} in the
graph-theoretic sense.
That is, if $G$ is the dual arrangement graph of a set of disks, and $A$ and $B$ are two subgraphs of $G$ corresponding to disks $R_A$ and $R_B$ respectively, 
then the
induced subgraphs on $V(A)\setminus V(B)$ and $V(B)\setminus V(A)$
are connected subgraphs of $G$.
If we are given a set $P$ of points and a set $D$ of disks, this defines a hypergraph $(P,\mathcal{D})$ on $P$ where each $D\in\mathcal{D}$ defines a hyperedge $D\cap P$.
For each cell containing a point of $P$,
we choose a representative point of $P$ from this cell as its representative vertex, and assign it a color $blue$.
Otherwise, the vertex corresponding to the cell is colored \emph{red}. 
A planar support for the hypergraph defined by blue vertices and disks then easily yields a planar support for the hypergraph $(P,\mathcal{D})$. See Fig.~\ref{fig:disksandpoints} for the construction of
a graph and a support.
This approach enables the study of geometric hypergraphs entirely within a graph-theoretic framework.

\begin{figure}[ht!]
\centering
\begin{minipage}[t]{0.3\textwidth}
\centering
\includegraphics[scale=.5]{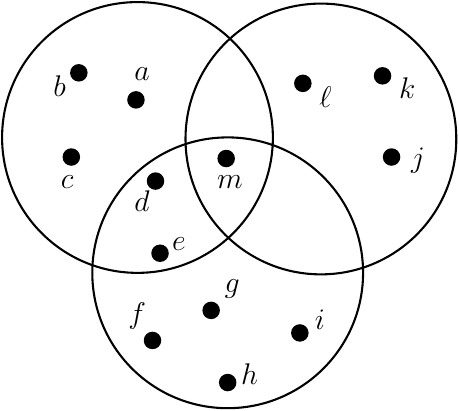}
 \\
 (a) Disks and Points
\label{fig:disksAndPoints}
\end{minipage}
\hfill
\begin{minipage}[t]{0.4\textwidth}
\centering
\includegraphics[scale=.5]{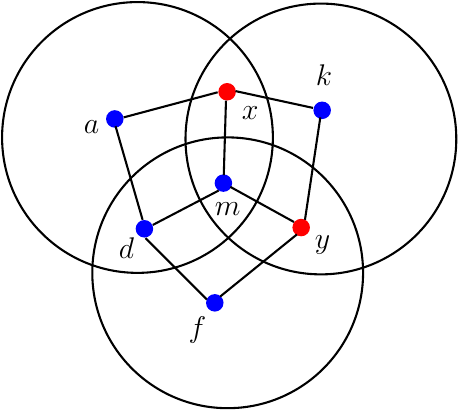}
\\
(b) Dual arrangement graph
\label{fig:DualArrangemenGraph}
 \end{minipage}\\
 \vspace{1cm}
 \begin{minipage}[t]{0.35\textwidth}
 \centering
\includegraphics[scale=.56]{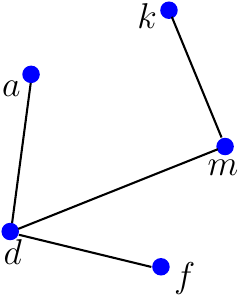}
\\
(c) Support on blue points for hypergraph in (b)
\label{fig:SupportOnBlue}
 \end{minipage}
 \hfill
 \begin{minipage}[t]{0.35\textwidth}
 \centering
\includegraphics[scale=.55]{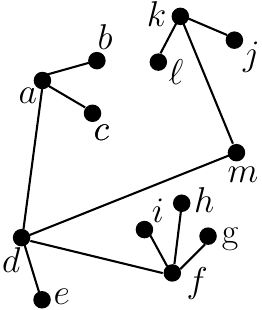}
\\
(d) Support for the hypergraph defined in (a)
\label{fig:SupportForHypergraph}
 \end{minipage}
 \caption{Support for hypergraph defined by disks and points in the plane.}\label{fig:disksandpoints}
 \end{figure}

To extend the results from the plane to other surfaces, we introduce the notion of \emph{cross-free} subgraphs of a graph
embedded on a surface of genus $g$.
We show that if the subgraphs of a host
graph are cross-free, then the \emph{primal, dual} and \emph{intersection}
hypergraphs admit a support of genus at most $g$. 

Our results on the existence of supports of bounded genus implies, via the
local-search framework, 
PTAS for the problems studied by~\cite{RR18} generalized to surfaces of bounded genus.
We also show that the non-piercing condition is not sufficient to guarantee
PTAS for hypergraphs on surfaces of bounded genus by showing an 
APX-hardness result for a covering problem defined by non-piercing regions on a torus. 

Besides the applications mentioned above, the existence of supports
are also useful for some coloring problems on hypergraphs. 
Finally, one can view the existence of a sparse support as a notion of sparsity for hypergraphs.
Thus, we could say that a given hypergraph is \emph{planar} if it admits a planar support, and
has bounded genus if it admits a support of bounded genus. While it may be hard to characterize
exactly the class of hypergraphs that are \emph{planar} in our sense, it is still interesting
to identify a large class of naturally defined hypergraphs that are indeed \emph{planar} or have
\emph{bounded genus}.
We can then view our result as identifying a class of naturally defined \emph{bounded genus}
hypergraphs.

\subsection{Problems studied}
\label{sec:problems}
The general class of problems we study are the following.
Let $G=(V,E)$ be a graph of genus $g$, with a set $\B(V)\subseteq V$ of \emph{terminals}. The remaining vertices, denoted $\R(V)$ are
the \emph{non-terminals}.
Let $\mathcal{H}$ be a collection of connected subgraphs of $G$. 
We denote by $(\B(V),\HH)$, the hypergraph on $\B(V)$ defined by $\HH$,
where each $H\in\HH$ defines a hyperedge $\B(V)\cap V(H)$. 
The question we study is whether the hypergraph $(\B(V),\HH)$ admits a support $Q$ on $\B(V)$
so that $Q$ also has genus at most $g$. 
We call this problem the \emph{primal support}
problem to distinguish it from the other notions of support we study.

\begin{problem}
  \problemtitle{Primal Support}
  \probleminput{A graph $G=(V,E)$ of genus $g$, $c:V\to\{\R,\B\}$, and a collection $\mathcal{H}$ of connected subgraphs of $G$.}
  \problemquestion{Is there a support $Q$ on $\B(V)$ of genus $g$, i.e., for each $H\in\mathcal{H}$,
  $H\cap\B(V)$ induces a connected subgraph in $Q$.}
\end{problem}

We also study the question of construction of a support for the
\emph{dual hypergraph} defined on the same input.
The dual hypergraph $(\mathcal{H},\{\mathcal{H}_b\}_{b\in \B(V)})$
is the hypergraph on $\mathcal{H}$, where each $b\in \B(V)$ defines a hyperedge $\mathcal{H}_b=\{H\in\mathcal{H}: b\in V(H)\}$
consisting of all $H\in\mathcal{H}$ containing $b$. A support
for the dual hypergraph is called a \emph{dual support}.
For the dual support, observe that we can assume $\B(V)=V$, as such a support
is also a support when $\B(V)$ is a proper subset of $V$.

\begin{problem}
      \problemtitle{Dual Support}
  \probleminput{A graph $G=(V,E)$ of genus $g$, and a collection $\HH$ of connected subgraphs of $G$.}
  \problemquestion{Is there a dual support $Q^*$ of genus $g$ on $\mathcal{H}$, i.e., for each $v\in V(G)$,
  $\HH_v=\{H\in\mathcal{H}: v\in H\}$ induces a connected subgraph of $Q^*$.}
\end{problem}

We now define \emph{intersection hypergraphs} that are a common generalization 
of primal and dual hypergraphs.
Let $G=(V,E)$ be a graph and let $\mathcal{H}$ and $\mathcal{K}$
be two families of connected subgraphs of $G$. 
An \emph{intersection hypergraph} is the hypergraph 
$(\mathcal{H},\{\mathcal{H}_K\}_{K\in\mathcal{K}})$, i.e., a 
hypergraph on $\mathcal{H}$ such that each $K\in\mathcal{K}$ defines
a hyperedge $\HH_K$ consisting of the subgraphs in $\mathcal{H}$ intersecting a
vertex of $K$. 
As in the dual setting,
choosing a subset of $\mathcal{H}$ as \emph{terminals} does not
add any additional complication, and hence we define the
intersection hypergraph on all of $\mathcal{H}$.
Setting either $\mathcal{H}$ or
$\mathcal{K}$ as singleton vertices yields respectively, the
primal and dual hypergraphs.
An \emph{intersection support} is a
support for the intersection hypergraph. 

\begin{problem}
      \problemtitle{Intersection Support}
  \probleminput{A graph $G=(V,E)$ of genus $g$, and two collections $\mathcal{H}$ and $\mathcal{K}$ of connected subgraphs of $G$.}
  \problemquestion{Is there an intersection support $\tilde{Q}$ on $\mathcal{H}$ of genus $g$, 
  i.e., for each $K\in\mathcal{K}$, the set $\HH_K=\{H\in\mathcal{H}: V(H)\cap V(K)\neq\emptyset\}$ induces a connected subgraph of $\tilde{Q}$.}
\end{problem}

Our main result is the following: We show that if the host graph has genus $g$ and the subgraphs in $\mathcal{H}$ and $\mathcal{K}$ are \emph{cross-free}, 
then there is an intersection support of genus $g$. 
This implies an equivalent result for the primal and dual hypergraphs. Fig.~\ref{fig:primaldualint} shows examples
of primal, dual and intersection supports.

\begin{figure}[ht!]
\centering
\begin{minipage}{0.3\textwidth}
\begin{center}
\includegraphics[scale=.5]{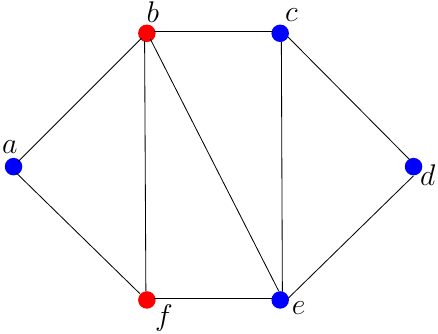}
\vspace{1mm}\\
(a) Primal hypergraph
 \label{fig:primalhyp}
\end{center}
\end{minipage}
 \begin{minipage}{0.3\textwidth}
 \begin{center}
\includegraphics[scale=.5]{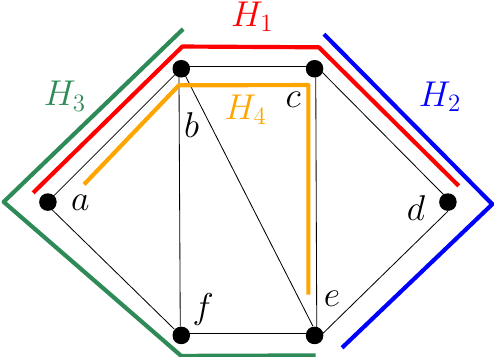}\\
(b) Dual hypergraph
\label{fig:dualhyp}
\end{center}
 \end{minipage}
 \begin{minipage}{0.35\textwidth}
 \begin{center}
\includegraphics[scale=.5]{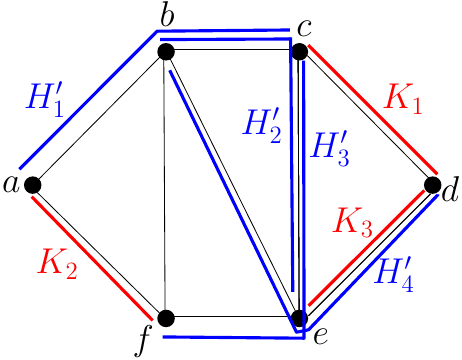}\\
(c) Intersection hypergraph
\label{fig:inthyp}
\end{center}
 \end{minipage}
 \vspace{5mm}\\
\begin{minipage}{0.27\textwidth}
\begin{center}
\includegraphics[scale=.55]{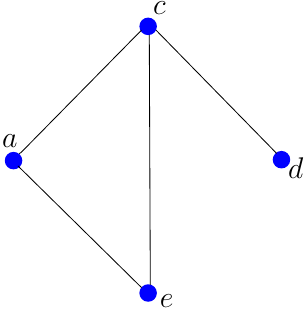}
\vspace{-1mm}\\
(d) Primal support
\label{fig:primalsup}
\end{center}
 \end{minipage}
 \begin{minipage}{0.33\textwidth}
 \begin{center}
\includegraphics[scale=.55]{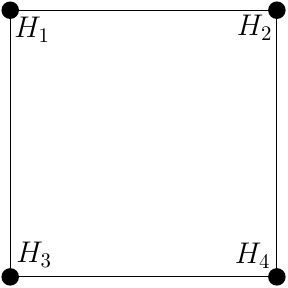}\\
(e) Dual support
\label{fig:dualsup}
\end{center}
 \end{minipage}
 \begin{minipage}{0.33\textwidth}
 \begin{center}
\includegraphics[scale=.55]{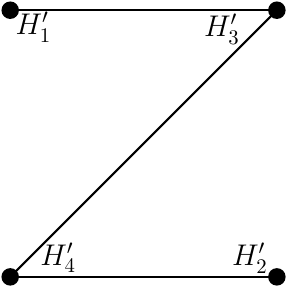}\\
(f) Intersection support
\label{fig:intsup}
\end{center}
 \end{minipage} 
 \caption{(a) and (b): Primal and Dual hypergraphs on the graph system $(G,\HH)$ where $G$ is the graph on vertices $\{a,b\ldots,f\}$, and $\HH=\{H_1,\ldots H_4\}$ with $H_1=\{a,b,c,d\}$, $H_2=\{c,d,e\}$, $H_3=\{a,b,f,e\}$, $H_4=\{a,b,c,e\}$; (c): Intersection hypergraph for $(G,\HH',\KK)$ with $\HH'=\{H'_1,\ldots,H'_4\}$ and $\KK=\{K_1,\ldots,K_3\}$, where $H'_1=\{a,b,c\}$, $H'_2=\{b,c,e\}$, $H'_3=\{c,e,f\}$, $H'_4=\{b,e,d\}$, and $K_1=\{c,d\}$, $K_2=\{a,f\}$, $K_3=\{d,e\}$. Figures (d)-(f) show the respective support graphs. Note: in (f), the pair of subgraphs $H'_3,H'_4$ is not \emph{cross-free} at vertex $e$ (see Definition \ref{def:cross_free}).}
 \label{fig:primaldualint}
 \end{figure}

\subsection{Related Work}
The notion of a planar analogue of a hypergraph was first suggested by~\cite{zykov}, who defined a hypergraph to be planar if there is a plane graph on the vertices of the hypergraph such that for each hyperedge, there
is a bounded face of the embedding containing only the elements of this hyperedge. Equivalently, a hypergraph is \emph{Zykov-planar} iff its incidence bipartite graph is planar.~\cite{voloshina1984planarity} introduced the notion of hypergraph planarity in the context of planarizing VLSI circuits (see the monograph~\cite{feinberg2012vlsi} and
references therein).
The notion is now called \emph{a planar support}.~\cite{Johnson1987Hypergraph} showed that the problem of deciding if a hypergraph has a planar support is NP-hard.

Since then, several authors have studied the question of deciding if a hypergraph admits a support from a restricted family of graphs.~\cite{tarjan1984simple} showed that we can decide in linear time if a hypergraph admits a tree support.~\cite{buchin2011planar} showed polynomial time algorithms to decide if a hypergraph admits a support that is a path, a cycle, or a tree with bounded maximum degree.
Further, the authors sharpen the result of~\cite{Johnson1987Hypergraph} by showing that deciding if a hypergraph admits a support that is a 2-outerplanar\footnote{A graph is 2-outerplanar if the graph can be embedded in the plane such that the vertices are on two concentric circles and removing all vertices of outer face results in an outerplanar graph.} graph is NP-hard.
The problems of constructing a support with the fewest number of edges, or with minimum maximum degree have also been studied in~\cite{baldoni2007tera,onus2011minimum}.

As stated earlier, the existence of supports and related notions have been used in the analysis of algorithms under a local-search framework
for packing and covering problems on geometric hypergraphs.~\cite{ChanH12} gave a PTAS for the Independent
Set problem for an arrangement of pseudodisks in the plane.~\cite{mustafa2010improved} gave a PTAS for the Hitting Set problem for a set of points and pseudodisks in the plane
(both results apply more generally to \emph{k-admissible regions} \footnote{A set of connected bounded regions $\mathcal{R}$ in the plane, each of whose boundary is a simple Jordan curve is $k$-admissible (for even $k$) if for any $R,R'\in\mathcal{R}$, $R\setminus R'$ and $R'\setminus R$ are connected, their boundaries are in general position, and 
intersect each other at most $k$ times. If $k=2$, the regions are called pseudodisks.} \cite{whitesides1990k}). 
The results of~\cite{ChanH12,mustafa2010improved} were extended by~\cite{BasuRoy2018} who 
used the same framework to design a PTAS for the Set Cover and Dominating Set problems for the intersection graph of pseudodisks.~\cite{RR18} showed the existence of a planar support for intersection hypergraphs of non-piercing regions
(see Section~\ref{sec:applications} for the definition), and their result gave a unified analysis for all packing and covering problems described
above. 
The local-search framework has also found applications in variants of the art gallery/terrain guarding problem,
where the geometric objects are defined implicitly as the visibility region of a guard, and the regions to be guarded are continuous regions, see for example~\cite{krohn2014guarding, DBLP:conf/wads/BandyapadhyayR17}.

~\cite{DBLP:journals/dcg/RamanR22}, showed that  
LP-rounding can be combined with a local-search framework to obtain 
$(2+\epsilon)$-approximations for the problem of Hitting Set and Set Cover with demands for hypergraphs defined on
points and pseudodisks. 

~\cite{CabelloG14} showed that Independent Set and Vertex Cover admit a PTAS on graphs excluding a fixed minor, and~\cite{DBLP:conf/walcom/AschnerKMY13} studied some packing and covering problems involving geometric non-planar graphs. 
Unlike their results, for the problems we study, we require additional work to show the existence of an appropriate local-search graph and this,
in some cases requires the construction of a support graph.

Besides packing and covering  problems, we also consider coloring problems on hypergraphs.~\cite{KellerS18} considered intersection hypergraphs of two families of disks, and showed that
such hypergraphs admit a \emph{conflict-free coloring}\footnote{In a conflict-free coloring, we want to color the vertices of the hypergraph
such that each hyperedge has a uniquely colored vertex.} with $O(\log n)$ colors.~\cite{Keszegh20} generalized their result to intersection hypergraphs of two families of pseudodisks and
showed that these hypergraphs admit a conflict-free coloring with $O(\log n)$ colors. The result holds for any hereditary family of hypergraphs if each hypergraph in the family admits a proper coloring using a constant number of colors, see~\cite{even_conflictfree}.
Therefore, for certain geometric hypergraphs defined on a surface of constant genus, our result in Section \ref{sec:color_geom} generalize the previous work of~\cite{even_conflictfree,KellerS18,Keszegh20,RR18}.

\section{Preliminaries}
\label{sec:preliminaries}

A graph $G=(V,E)$ and a collection of subgraphs $\mathcal{H}$ of $G$ naturally defines a hypergraph 
$(V(G),\{V(H): H\in\mathcal{H}\})$, where $V(H) = \{v\in V: v\in H\}$. We call the tuple $(G,\mathcal{H})$ a \emph{graph system} which we use while considering primal and dual hypergraphs.
Similarly, given two collections $\mathcal{H}$ and $\mathcal{K}$ of connected subgraphs of $G$, we call $(G,\mathcal{H},\mathcal{K})$
an \emph{intersection system}, and we use it while considering an intersection hypergraph. We implicitly make the assumption that $G$ is a connected graph.
We assume throughout that the subgraphs of
the host graph are connected, induced subgraphs.
Note that assuming that the subgraphs are induced is without loss of generality
as the hypergraphs are defined on the vertices of $G$.
We also assume that the graph $G$ is given along with a \emph{2-cell embedding} in an oriented surface. Below, we define these notions.

\begin{definition}[Embedding of a graph]
A graph $G$ is said to be \emph{embedded} on a surface $\Sigma$ if the vertices of $G$ are distinct points on $\Sigma$ and each edge of $G$ is a simple arc lying on $\Sigma$ whose endpoints are the vertices of the edge, and such that its interior is disjoint from other edges and vertices.
A 2-cell embedding is an embedding of a graph on a surface, where each face is homeomorphic to a disk in the plane.
\end{definition}

We say that a graph $G$ has an \emph{embedding} on a surface $\Sigma$ if there is a graph $G'$ embedded on $\Sigma$ such that $G'$ is isomorphic to $G$.
An orientable surface has genus $g$ if it is obtained from a sphere by adding $g$ \emph{handles} (See~\cite{Mohar2001GraphsOS}, Chapter 3).

\begin{definition}[Genus] The \emph{genus} $g$ of a graph $G$ is the minimum genus of an oriented surface $\Sigma$ so that $G$ has an embedding on $\Sigma$.

\end{definition}
We say that a graph has bounded genus if it can be embedded on a surface whose genus is bounded. 
It should be noted that contracting any edge of a graph does not increase the genus of the resulting graph and we will use this fact throughout the paper.

For a graph $G=(V,E)$ a set $S\subset V$ is a \emph{separator} if $G\setminus S$ consists of two disjoint components $A$ and $B$.
A separator is said to have \emph{sub-linear size} if $|S|=O(n^{\delta})$ for a constant $\delta<1$. A separator is said to be \emph{balanced}
if $|A|, |B|\le\alpha|V|$ for some constant $\alpha < 1$. We also need the notion of a \emph{weighted separator}. Let $w:V\to\mathbb{R}_{\ge 0}$
be a weight function on the vertices. A separator $S$ is said to be balanced if $w(A), w(B)\le \alpha w(V)$, where $w(X)=\sum_{v\in X} w(v)$.
For the separator $S$, its size is still its cardinality.
The key property of graphs of bounded genus we use is that they have sub-linear sized balanced separators. In the rest of the paper, since we only
deal with sub-linear balanced separators, we will just use the term separator to mean a sub-linear sized, balanced separator.

\begin{theorem}[\cite{gilbert1984separator}]
\label{thm:gilbert}
Let $G=(V,E)$ be an $n$-vertex graph of genus $g$. Let $w:V\to\mathbb{R}_{\ge 0}$ be a weight function.
Then, $G$ has a balanced separator of size $O(\sqrt{gn})$ such that $w(A), w(B) \le \frac{1}{2}w(V)$.
\end{theorem}

For a vertex $v\in V(G)$, we use $N_G(v)$ to denote the neighbors of $v$ in $G$ 
(or just $N(v)$ if $G$ is clear from context). 
We use $u\sim v$ or $e\sim v$ to denote respectively, if $u$ and $v$ are adjacent in $G$, 
or an edge $e$ incident to vertex $v$. For $S \subseteq V$, we use $G[S]$ to denote the subgraph of $G$ induced on $S$.
We use $\mathcal{H}_v = \{H\in\mathcal{H}: v\in H\}$.
Similarly, let $\mathcal{H}_e=\{H\in\mathcal{H}: e\in H\}$.
We let $\depth(v)=|\{H\in\mathcal{H}: v\in H\}|$. Similarly, let $\depth(e)=|\{H\in \mathcal{H}: e\in H\}|$.
We also use the notations $H\cap K$, $H\setminus K$, and $H\subseteq K$ to mean
$V(H)\cap V(K)$, $V(H)\setminus V(K)$, and $V(H)\subseteq V(K)$ respectively for any two subgraphs $H,K$ of a graph $G$.

Our goal is to consider restrictions on the subgraphs so that the support is guaranteed to have bounded genus. 
To that end, we introduce a notion of cross-free hypergraphs and non-piercing hypergraphs.

\begin{definition}[Reduced graph]\label{def:reducedgraph}
Let $(G,\mathcal{H})$ be a graph system. 
For two subgraphs $H,H'\in\mathcal{H}$, the \emph{reduced graph} $R_G(H,H')$ (or just $R(H, H')$ if $G$ is clear from context) is the graph obtained from $G$ by contracting all edges, both of whose end-points are in $H\cap H'$. 
\end{definition}

Note that if $G$ is embedded on a surface $\Sigma$, then this induces an embedding of $R_G(H,H')$ on $\Sigma$.

\begin{definition}[Cross-free at $v$]\label{def:cross_free}
Let $(G,\mathcal{H})$ be an embedded graph system. 
Two subgraphs $H,H'\in\mathcal{H}$ are said to be cross-free at a vertex $v\in V(H)\cap V(H')$ 
if the following holds: Consider the induced embedding of the reduced graph $R(H,H')$.
Let $\tilde{v}$ be the image of $v$ in $R(H,H')$. There are no 4 edges 
$e_i=\{\tilde{v},v_i\}$ in $R(H,H')$, $i=1,\ldots, 4$ incident to $\tilde{v}$ in cyclic order around $\tilde{v}$,
such that $v_1, v_3\in H\setminus H'$, and $v_2, v_4\in H'\setminus H$.
\end{definition}

For an embedded graph system $(G,\mathcal{H})$, if every pair $H, H'\in\mathcal{H}_v$ is cross-free at $v$, then
$(G,\mathcal{H})$ is said to be cross-free at $v$. For two subgraphs $H, H'$ if $v$ is not contained in both $H$ and $H'$,
then $H$ and $H'$ are assumed to be cross-free at $v$.  
The embedded graph system $(G,\mathcal{H})$ is cross-free if it is cross-free at all $v\in V(G)$.
Finally, a graph system $(G,\mathcal{H})$ is cross-free if there exists an embedding of $G$ 
such that the embedded graph system $(G,\mathcal{H})$ is cross-free with respect to $\mathcal{H}$.
If there exist $H, H'\in \mathcal{H}$ such that $H$ and $H'$ are not cross-free at $v$, we say that $H$ and $H'$ are \emph{crossing} at $v$.
Throughout the paper, 
when we say that $G$ is embedded on a surface, we assume without loss of generality that the embedding is cross-free with respect to $\mathcal{H}$.

An intersection hypergraph $(G,\mathcal{H},\mathcal{K})$ is cross-free
if there is an embedding of $G$ such that the embedded graph systems $(G,\mathcal{H})$ and $(G,\mathcal{K})$ are 
simultaneously cross-free.
Note that we can have $H\in\mathcal{H}$, $K\in\mathcal{K}$ that are crossing.
Finally, we use the term $(G,\mathcal{H})$ is a cross-free system of genus $g$ or the term
$(G,\mathcal{H},\mathcal{K})$ is a cross-free intersection system of genus $g$ to mean that the host
graph $G$ has genus $g$.

\begin{definition}[Non-piercing subgraphs]
\label{defn:nonpiercing}
A graph system $(G,\mathcal{H})$ with $\mathcal{H}$ a collection of subgraphs of $G$
is non-piercing if each $H\in\mathcal{H}$ is connected and for any two 
subgraphs $H, H'\in\mathcal{H}$, $H\setminus H'$ induces a connected
subgraph of $G$.
\end{definition}

Note that non-piercing is a purely combinatorial notion, 
and unlike the cross-free property above, it does not require an embedding of the graph. 
If $\exists\; H, H'\in\mathcal{H}$ such that either the induced subgraph $H\setminus H'$ or the induced subgraph $H'\setminus H$ is not connected, then we say that $H$ and $H'$ are \emph{piercing}.

\smallskip
\noindent

\smallskip
\noindent

\section{Contribution}
\label{sec:contribution}
The main result we show in this paper is that 
if $(G,\mathcal{H},\mathcal{K})$ is a cross-free intersection system of genus $g$, then there
is an intersection support $\tilde{Q}$ of genus at most $g$. 
In order to show that $\tilde{Q}$ exists, we show that primal and dual supports of genus at most $g$ exist for
cross-free graph systems of genus $g$. These results are a generalization of the results of~\cite{RR18}, 
who showed similar results for non-piercing regions in the plane. While some of the techniques in the two papers are
similar, several new ideas are required for the results to go through. Besides generalization to surfaces, 
our results have the advantage that they use basic graph operations and do not require delicate topological arguments
as in~\cite{RR18}. We also give applications of our results to several packing, covering problems and coloring problems.

\subsection{Existence of Supports}

We obtain the following results for the construction of primal, dual, and intersection supports for cross-free systems.

\begin{restatable}{thm}{primalsupport}
\label{thm:primalsupport}
Let $(G,\mathcal{H})$ be a cross-free graph system of genus $g$.
For any coloring $c:V(G)\to\{\R,\B\}$, there exists a primal support $Q$ of genus at most $g$ on $\B(V)$.
\end{restatable}

\begin{restatable}{thm}{dualembedded}
\label{thm:dualembedded}
Let $(G,\mathcal{H})$ be a cross-free graph system of genus $g$. Then, there exists a dual support $Q^*$ on $\mathcal{H}$
of genus at most $g$.
\end{restatable}

\begin{restatable}{thm}{intsupport}
\label{thm:intsupport}
Let $(G,\mathcal{H},\mathcal{K})$ be a cross-free intersection system of genus $g$. 
Then, there exists an intersection support $\tilde{Q}$ on $\mathcal{H}$ of genus at most $g$.
\end{restatable}

While the three results above are algorithmic, we are unable at this point to prove that the algorithms to construct the supports run in polynomial time, and we leave this as an open problem.

\subsection{Algorithms for Packing and Covering Problems}
As stated earlier, our motivation for constructing supports comes from the analysis of local-search
algorithms for a family of packing and covering problems. The fact that the graph we construct is a support, satisfies one of the requirements for the analysis of the algorithms.
The second requirement is that the graph should belong to a hereditary family with sublinear sized separators.
Since graphs of bounded genus have sublinear sized separators as shown by~\cite{alon1990separator, gilbert1984separator}, our results above, combined with the framework
in~\cite{ChanH12, mustafa2010improved, RR18} directly imply the following results
generalizing the results of~\cite{RR18} in the plane to higher genus surfaces.

In the \emph{generalized capacitated packing problem}, the input is an intersection system $(G,\mathcal{H},\mathcal{K})$ with 
a capacity function $\cp:\mathcal{K}\to\mathbb{N}$. The goal is to find a largest sub-collection $\mathcal{H}'\subseteq\mathcal{H}$ so that
$|\mathcal{H}'_K|\le\cp(K)$ for all $K\in\mathcal{K}$.
We prove the following.

\begin{restatable}{thm}{crossfreePTAS}
\label{thm:crossfreePTAS}
Let $(G,\mathcal{H},\mathcal{K})$ be a cross-free intersection system of genus $g$ for some constant $g$, and 
let $\cp:\mathcal{K}\to\mathbb{N}$ be a capacity function.
Let $\Delta>0$ be an absolute constant. 
If $cap(K)\le\Delta$ for all $K\in\mathcal{K}$,
then the {\bf Generalized Capacitated Packing Problem} admits a PTAS via the local-search framework.
\end{restatable}

For a graph system $(G,\mathcal{H})$, let $\cp:V(G)\to\mathbb{N}$ be a capacity function on $V(G)$.
The {\bf Capacitated} $\mathbf{\mathcal{H}}${\bf-packing} problem is to find a maximum cardinality collection $\mathcal{H}'\subseteq\mathcal{H}$
such that for each $v\in V(G)$,
$|\mathcal{H}'_v|\le\cp(v)$.

Let $\cp:\mathcal{H}\to\mathbb{N}$ be a capacity function
on $\mathcal{H}$. Then the {\bf Capacitated Vertex Packing} problem is to find
the largest cardinality subset $V'\subseteq V(G)$ s.t. 
$|H\cap V'|\le\cp(H)$ for all $H\in\mathcal{H}$.

\begin{restatable}{cor}{packpack}
\label{cor:packpack}
Let $(G,\mathcal{H})$ be a cross-free graph system of genus $g$ for some constant $g$,
and let $\Delta>0$ be an absolute constant.
\begin{enumerate}
\item Given a capacity
function $\cp:V(G)\to\mathbb{N}$ such that 
$\cp(v)\le\Delta$
for all $v\in V$.
Then, there is a PTAS for the 
{\bf Capacitated $\mathcal{H}$-Packing}
problem via the local-search framework.
\item Given a capacity
function $\cp:\mathcal{H}\to\mathbb{N}$ such that $\cp(H)\le\Delta$ for all $H\in\mathcal{H}$. 
Then there is a PTAS for the {\bf Capacitated Vertex Packing} 
problem via the local-search framework.
\end{enumerate}
\end{restatable}

\begin{restatable}{thm}{intcover}
\label{thm:intcover}
Let $(G,\mathcal{H},\mathcal{K})$ be a cross-free intersection system
of genus $g$ where $g$ is bounded above by a constant. Then, there is a PTAS for the 
{\bf Generalized Covering Problem} via the local-search framework.
\end{restatable}

For a graph system $(G,\mathcal{H})$,
an \emph{intersection graph} is the graph 
$D=(\mathcal{H},E)$, where 
$\{H,H'\}\in E(D)$ if and only if $V(H)$ and $V(H')$
intersect at a vertex of $G$. 

\begin{restatable}{cor}{dominatingcrossfree}
\label{cor:dominating}
Let $(G,\mathcal{H})$ be a cross-free graph system of genus $g$ for some constant $g$, and $\mathcal{D}$ denote the intersection graph.
Then, the following problems on $\mathcal{D}$ admit a PTAS:
$(a)$ Independent Set, $(b)$ Vertex Cover, and $(c)$ Dominating Set.
\end{restatable}

We show a natural class of geometric graphs and hypergraphs that can be modeled as 
cross-free graph systems. Hence, the results on Packing and Covering problems above hold 
for such hypergraphs.

A set $\mathcal{D}$ of connected regions in the plane is said to be \emph{weakly non-piercing} if $D\setminus D'$ or
$D'\setminus D$ is connected for all $D,D'\in\mathcal{D}$. This is weaker than the \emph{non-piercing} condition considered in~\cite{RR18}, 
which requires that both $D\setminus D'$ and $D'\setminus D$ are connected.

\begin{restatable}{thm}{touchcross}
\label{thm:touchcross}
Let $\mathcal{H}$ and $\mathcal{K}$ be two families of simply connected weakly non-piercing regions on an oriented surface $\Sigma$ of genus $g$.
Then, we can define an embedded cross-free intersection system $(G,\mathcal{H}',\mathcal{K}')$ of genus $g$ and a bijection between
$\mathcal{H}$ and $\mathcal{H}'$, and between $\mathcal{K}$ and $\mathcal{K}'$ so that the intersection hypergraph defined by $\mathcal{H}$ and $\mathcal{K}$
is isomorphic to that defined by $\mathcal{H}'$ and $\mathcal{K}'$.
\end{restatable}

We also show that covering problems involving non-piercing
subgraphs of a bounded genus graph are APX-hard. For non-piercing regions in the plane, these problems admit a PTAS.

\begin{restatable}{thm}{apxhard}
\label{thm:apxhard}
There exist crossing non-piercing graph systems $(G,\mathcal{H})$ with $G$ embedded on the torus such that the Set Cover problem is APX-hard. Similarly, the Hitting Set problem on such a set system is APX-hard.
\end{restatable}

\subsection{Hypergraph Coloring}
We obtain the following results for the hypergraph coloring problem, generalizing the results of~\cite{KellerS18},~\cite{Keszegh20} and~\cite{RR18}. Our results in combination with previous results also imply a conflict-free coloring
with $O(\log n)$ colors.

\begin{restatable}{thm}{colorhypergraph}
\label{thm:colorhypergraph}
Let $\mathcal{H}$ and $\mathcal{K}$ be two families of simply connected weakly non-piercing regions on an oriented surface of genus $g$.
Then the intersection hypergraph has a proper coloring with at most $\frac{7 + \sqrt{1+24g}}{2}$ colors i.e.,
$\mathcal{H}$ can be colored with at most $\frac{7 + \sqrt{1+24g}}{2}$ colors such that for any $K\in\mathcal{K}$,
no hyperedge ${\mathcal{H}_K}$ is monochromatic. 
\end{restatable}
 
\smallskip\noindent
{\bf A guide to the reader:}
In Section~\ref{sec:nonpiercingimpliescrossfree}, we contrast non-piercing graph
systems with cross-free graph systems. In Section~\ref{sec:vertexbypassing}, we introduce the notion of \emph{Vertex Bypassing}, a basic tool used
in the construction of supports. In Section~\ref{sec:constructiongenus}, we show how the primal, dual and intersection supports can be constructed.
We describe applications of our results on the existence and construction of supports in Section~\ref{sec:applications}.
We conclude in Section~\ref{sec:conclusion} with open questions.

\section{Non-piercing and Cross-free Systems}
\label{sec:nonpiercingimpliescrossfree}
In this section, we relate the geometric notion of non-piercing regions to the graph-theoretic notion of being cross-free.
We show that the non-piercing condition implies the cross-free condition in the plane, but they are incomparable in higher genus surfaces.
Below, we define non-piercing regions and the dual arrangement graph of regions.

\begin{definition}[Non-piercing regions]\label{def:npregions}
Let $\mathcal{R}$ be a set of compact, path connected regions in general position in $\mathbb{R}^2$ such that boundary of each 
$R\in\mathcal{R}$ consists of a finite set of disjoint simple Jordan curves.
Then $\mathcal{R}$ is said to be non-piercing if for all $R ,R'\in\mathcal{R}$, the regions $R\setminus R'$ and $R'\setminus R$ are connected.
\end{definition}

By general position above, we mean that the boundaries of any two regions intersect at only a finite number of points where they cross, 
and that there are no three regions whose boundaries contain a common point. See Fig.~\ref{fig:pseudodisks, piercing and non-piercing} for an example.\\
Let $\mathcal{R}$ be an arrangement of connected and bounded regions in the plane, and $\partial(R)$ defines the boundary of $R\in\mathcal{R}$.
Then $\mathcal{R}\setminus\bigcup_{R\in\mathcal{R}}\partial(R)$ splits the plane into a set $\mathcal{C}$ of connected components, each called a \emph{cell} in the arrangement $\mathcal{R}$.
Then the dual arrangement graph $G=(\mathcal{C},E)$ of $\mathcal{R}$ is a plane graph with vertex set $\mathcal{C}$ where two vertices are adjacent if and only if their corresponding cells are separated by boundary of some region in $\mathcal{R}$.

Each region $R\in\mathcal{R}$ induces a connected subgraph $G_R$ of $G$ - consisting of vertices corresponding to cells in $R$.
Moreover, it is easy to see that if $\mathcal{R}$
is non-piercing, and $\mathcal{R}_G$ is the collection of induced subgraphs of $G$ each corresponding to a region in $\mathcal{R}$,
then the subgraphs in $\mathcal{R}_G$ are non-piercing subgraphs of $G$.

We start with the following result, which shows that if a graph system is non-piercing in the plane, it is cross-free.

\begin{theorem}
\label{thm:planenp}
Let $(G,\mathcal{H})$ be a planar non-piercing system, 
then, $(G,\mathcal{H})$ is cross-free.
\end{theorem}

\begin{proof}
We show that if $(G,\mathcal{H})$ is crossing, then it cannot be non-piercing.
Consider an embedding of $G$ in the plane. Abusing notation, let $G$ also denote the embedding
of $G$ in the plane, which we also refer to as $G$, abusing notation slightly.
If $(G,\mathcal{H})$ is crossing, there are two subgraphs 
$H, H'\in\mathcal{H}$ and a vertex $x\in H\cap H'$ in the reduced graph $R_{G}(H, H')$ that has four neighbors
$x_1, \ldots, x_4$ in cyclic order around $x$ such that $x_1, x_3\in H\setminus H'$ and
$x_2, x_4\in H'\setminus H$. It cannot be that both $x_1=x_3$, and $x_2=x_4$ 
without violating planarity. 
So assume without loss of generality that $x_2\neq x_4$.

Since $\mathcal{H}$ is non-piercing, $H$ is connected and $H\setminus H'$ induces a connected subgraph of $G$. Further,
note that $H$ and $H'$ are non-piercing in $G$, then they remain non-piercing in $R_G(H, H')$.
Therefore, there is an $x_1$-$x_3$ path $P$ in $R_G(H, H')$ that lies in $H\setminus H'$. 
Again, since $\mathcal{H}$ is non-piercing, $H'\setminus H$ induces a connected subgraph of $G$. 
Therefore, there is a path $P'$ between $x_2$ and $x_4$ that lies in $H'\setminus H$.
Observe that $P\cup\{x_1,x\}\cup\{x,x_3\}$
induces a Jordan curve with $x_2$ and $x_4$ on either side of it. 
Thus $P$ and $P'$ intersect at a vertex that lies in $H\cap H'$, which is not possible since $P$ and $P'$ are disjoint. Therefore, there is no path $P'$ between $x_2$ and $x_4$ in $H'\setminus H$
which implies $H'\setminus H$ is not connected
and thus $\mathcal{H}$ is piercing.
\end{proof}

Note that the reverse implication does not hold. It is easy to construct examples of graph systems in the plane that are cross-free, but are not non-piercing. 
Consider the graph system consisting of a graph $K_{1,4}$ embedded in the plane,
with central vertex $v$, and leaves $a,b,c,d$ in cyclic order.
Let $H$ and $H'$ be two subgraphs where $H$ is the graph induced on $\{v,a,b\}$ and $H'$ is the graph induced on $\{v,c,d\}$. Then, $H$ and $H'$ are cross-free, but neither $H\setminus H'$ nor $H'\setminus H$ is connected.

The proof of Theorem~\ref{thm:planenp} relies on the Jordan curve theorem (see Ch. 2 in~\cite{Mohar2001GraphsOS}), and
the corresponding statement need not hold for surfaces of higher genus. For example, let $G$ be the
torus grid graph $T_{n,n} = C_n \Box C_n$, see~\cite{weisstein2016torus}.  
The subgraphs $\mathcal{H}$ are the $n$ non-contractible cycles perpendicular to the hole, 
and the $n$ non-contractible cycles parallel to the hole. Note that the system $(T_{n,n},\mathcal{H})$ is non-piercing
but not cross-free.
Any pair of parallel and perpendicular cycles intersect at a unique vertex, and therefore in the dual support,
the vertices corresponding to these two cycles must be adjacent. Therefore, 
the dual support is
$K_{n,n}$ which is not embeddable on the torus for large enough $n$.

\section{Vertex Bypassing}
\label{sec:vertexbypassing}
The basic tool we use in the construction of 
supports is the Vertex Bypassing (\VB$(.)$) operation.
Given a cross-free embedding of $G$ with respect to
$\mathcal{H}$ and a vertex $v\in V(G)$, $\VB(v)$
removes $v$ and modifies the neighborhood around $v$ 
so that the
resulting system remains cross-free and is in 
some sense \emph{simpler} than the original system.

\begin{definition}[\VB$(v)$]
\label{defn:vertexbypassing}
Let $(G,\mathcal{H})$ be a cross-free system with a cross-free embedding of $G$ with respect to 
$\mathcal{H}$ on an oriented surface $\Sigma$.
Let $N(v)=(v_0,\ldots, v_{k-1}, v_0)$ be the cyclic order of neighbors of $v$ in the embedding. The Vertex Bypassing
operation involves the following sequence of operations:
\begin{enumerate}
\item Subdivide each edge $\{v,v_i\}$ by a vertex $u_i$.
Remove the vertex $v$ and connect consecutive vertices
$u_i, u_{i+1}$ (with indices taken~$\mathrm{mod}~ k$) with a simple arc not intersecting the edges of $G$ to
construct a cycle $C=(u_0,\ldots, u_{k-1}, u_0)$
so that the resulting graph $G''$ remains embedded on $\Sigma$.
\label{step:one}
\item For each $H\in \mathcal{H}_v$, let $H'$ denote the
subgraph of $G''$ induced by $(V(H)\setminus \{v\})\cup \{u_i: \{v, v_i\}\in H\}$. 
Let $\mathcal{H}'_v =\{H': H\in\mathcal{H}_v\}$.
Let $\mathcal{H}' = (\mathcal{H}\setminus\mathcal{H}_v)\cup\mathcal{H}'_v$ (Note that the subgraphs in $\mathcal{H}'_v$ may not be connected).\label{step:two}
\item Add a set $D$ of internally non-intersecting chords in $C$ so that $\forall H\in \mathcal{H}'$,
$H$ induces a connected subgraph in $C\cup D$, and the resulting graph system $(G',\mathcal{H}')$ remains cross-free. 
\label{step:three}
\end{enumerate}
\end{definition}

\begin{figure}[ht!]   
    \centering
    \begin{minipage}{0.32\textwidth}
    \begin{center}
    \includegraphics[scale=.46]{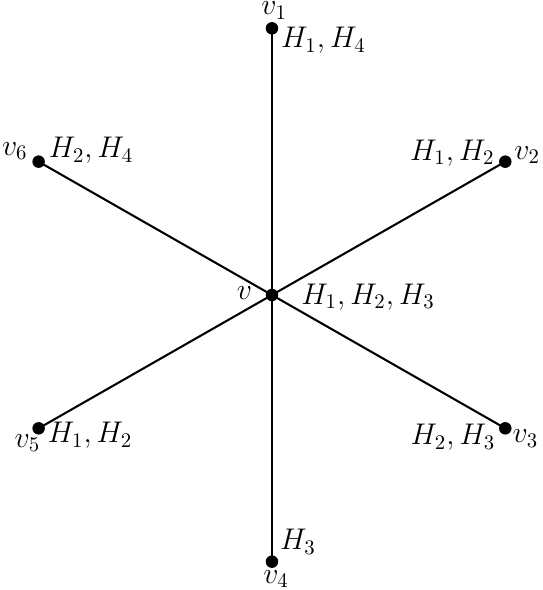}\\
    (a) Vertices in $N(v)$
            \label{vb1}
    \end{center}
    \end{minipage}
    \begin{minipage}{0.32\textwidth}
    \begin{center}
    \includegraphics[scale=.46]{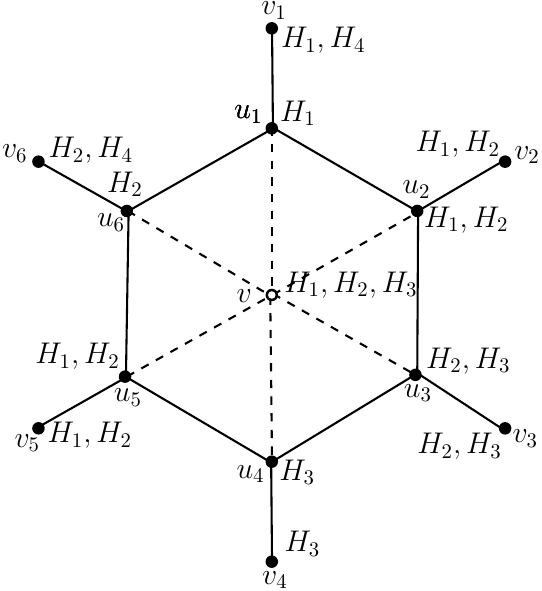}\\
    (b) Cycle $C$ on $u_i$'s
            \label{vb2}
    \end{center}
     \end{minipage}
     \begin{minipage}{0.32\textwidth}
    \begin{center}
    \includegraphics[scale=.46]{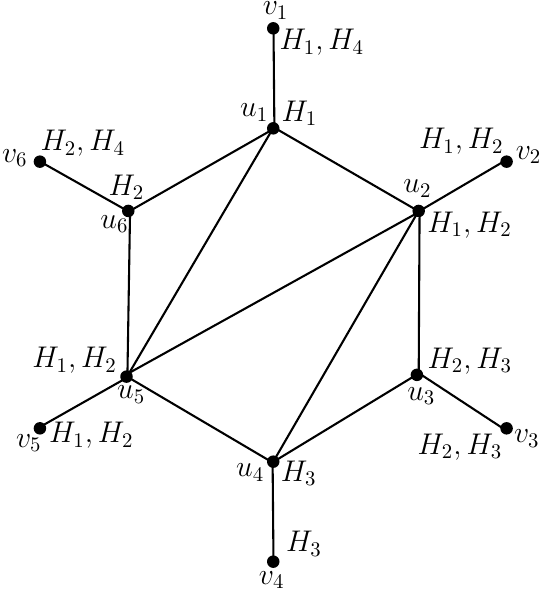}\\
    (c) Chords in $C$
    \label{fig:vb3}
    \end{center}
     \end{minipage}
     \caption{Vertex Bypassing.}\label{fig:vertexbypassing}
     \end{figure}

Fig.~\ref{fig:vertexbypassing} shows vertex bypassing at vertex $v$.
It is easy to see that the graph $G'$ obtained from $G$ 
is also embedded on $\Sigma$ as each operation preserves
the embedding.
Since we remove the vertex $v$, at the end of Step~\ref{step:two}, the subgraphs $\mathcal{H}'_v$ in $G''$ 
may be disconnected. The main challenge is to add additional edges to the graph $G''$ so that subgraphs $\HH'$ of $G'$ remain connected and the graph system $(G',\HH')$ remains cross-free.

In order to do so, we introduce the notion of $abab$-free hypergraphs.
An equivalent notion, namely $ABAB$-free hypergraphs was studied by~\cite{ackerman2020coloring}, where the elements
of the hypergraph are placed in a linear order instead of a cyclic order.

\begin{definition}[$abab$-free]
\label{defn:abab}
An abstract hypergraph $(X,\mathcal{H})$ is said to be $abab$-free if the elements of $X$ can be arranged on a cycle $C$ such that for any $H,H'\in\mathcal{H}$
there does not exist distinct $x_1, x_2, x_3, x_4\in X$ in this cyclic order around $C$ such that $x_1,x_3\in H\setminus H'$,
and $x_2,x_4\in H'\setminus H$.
\end{definition}

The relation between cross-free graph systems and
$abab$-free systems is the following:
If $G$ is embedded so that it is cross-free with respect to
$\mathcal{H}$, then for any vertex $v$, on applying $\VB(v)$,
the system $(C,\mathcal{H}'_v)$ is $abab$-free, where
$C$ is the cycle on the vertices $u_0,\ldots, u_{deg(v)-1}$
subdividing the edges $\{v,v_i\}$.

\begin{proposition}
\label{prop:crosabab}
Let $(G,\mathcal{H})$ be a cross-free graph system with
a cross-free embedding with respect to $\mathcal{H}$.
For any vertex $v\in V(G)$, on applying $\VB(v)$, 
$(C,\mathcal{H}'_v)$ is $abab$-free.
\end{proposition}
\begin{proof}
Since the subgraphs in $\mathcal{H}$ are induced subgraphs, $\mathcal{H}_e\subseteq\mathcal{H}_v$. 
Let $C=(u_0,\ldots, u_{deg(v)-1})$ be the cycle on the subdivided vertices added
on applying $\VB(v)$. 
By construction $\mathcal{H}_{u_i}=\mathcal{H}_{e}$ where $e=\{v,v_i\}$, and
$\mathcal{H}_e\subseteq\mathcal{H}_v$. If $(C,\mathcal{H}'_v)$ is not $abab$-free,
let $H_1, H_2$ be subgraphs in $\mathcal{H}'_v$ that are not $abab$-free.
This implies that $H_1$ and $H_2$ are crossing at $v$.
\end{proof}

By Proposition~\ref{prop:crosabab} therefore, 
the problem of adding a set of non-intersecting chords $D$ in Step~\ref{step:three} of Vertex Bypassing 
reduces to the following: Given an $abab$-free embedding of an $abab$-free hypergraph, can we add a set of non-intersecting
chords in $C$ such that each subgraph is connected? We show in the following lemma, whose proof is in Section~\ref{sec:nonblocking}
that we can always add such chords.

\begin{restatable}{lem}{nblockchord}
\label{lem:nblock}
Let $C$ be a cycle embedded in the plane, and let $\mathcal{K}$ be a set of $abab$-free subgraphs of $C$.
Then, we can add a set $D$ of non-intersecting chords in $C$ such that each $K\in\mathcal{K}$ 
induces a connected subgraph of $C\cup D$. Further, the set $D$ of non-intersecting chords to add can be computed in time $O(mn^4)$ where $m=|\KK|$, and $n=|C|$.
\end{restatable}

With Lemma~\ref{lem:nblock} at hand, we can obtain the desired system $(G',\mathcal{H}')$.

\begin{lemma}
\label{lem:keylem}
Let $(G,\mathcal{H})$ be a cross-free system with a cross-free embedding of $G$ with respect to $\mathcal{H}$. 
Suppose we apply \VB$(v)$ to a vertex $v\in V(G)$. Then, each subgraph $H$ in $(G',\mathcal{H}')$ is connected. Further,
\VB$(v)$ can be done in time $O(|\HH||V(G)|^4)$.
\end{lemma}
\begin{proof}
Let $C$ be the cycle added on the subdividing vertices around vertex $v$. Since $(G,\mathcal{H})$ is cross-free, 
by Proposition~\ref{prop:crosabab},
the subgraphs $\{H\cap C: H\in\mathcal{H}'_v\}$ 
satisfy the $abab$-free property on $C$. Therefore, by Lemma~\ref{lem:nblock}, there is a collection $D$ of non-intersecting chords such that 
each subgraph in $\mathcal{H}'_v$ induces a connected subgraph of $C\cup D$. Hence, each subgraph $H\in \mathcal{H}'$ is a connected subgraph of $G'$ since each $H\in\mathcal{H}_v$ is modified only in the vertices of subdivision. Since Lemma~\ref{lem:nblock} guarantees that the set $D$ of 
non-intersecting chords to add can be computed in time $O(|\HH||V(G)|^4)$, it follows that \VB$(v)$ can be done in $O(|\HH||V(G)|^4)$ time.
\end{proof}

In the following, we argue that if $(G,\mathcal{H})$ is cross-free, then the resulting system $(G',\mathcal{H}')$ obtained
on applying $\VB(v)$ at a vertex $v\in V(G)$ is cross-free.

\begin{lemma}
\label{lem:remcrossfree}
Let $(G,\mathcal{H})$ be a cross-free system and let
$v\in V(G)$. Let $(G',\mathcal{H}')$ be the system obtained on applying $\VB(v)$. Then, $(G',\mathcal{H}')$ is cross-free.
\end{lemma}
\begin{proof}
By Lemma \ref{lem:keylem}, the subgraphs in $\HH'$ are connected subgraphs of $G'$.
We argue that $(G',\mathcal{H}')$ is cross-free.
Let the cyclic sequence of the neighbors of $v$ be $(v_0,\ldots, v_{k-1})$.
Similarly, let $(u_0,\ldots, u_{k-1})$ denote the cyclic sequence of the vertices in $G'$
where $u_i$ is the vertex subdividing the edge $\{v,v_i\}$ in $\VB(v)$.

Consider two subgraphs $H_1, H_2\in\mathcal{H}$. We will show that $H'_1$ and $H'_2$ are cross-free,
where $H'_1$ and $H'_2$ are respectively, the subgraphs in $\HH'$ corresponding to $H_1$ and $H_2$ in $\mathcal{H}$.
For a vertex $x\in V(G)$, let $x'$ denote its copy in $V(G')$.
We let $\tilde{x}$ denote its corresponding vertex in the reduced graph $R_G(H_1, H_2)$,
and $\tilde{x'}$, the vertex in $R_{G'}(H'_1, H'_2)$ corresponding to vertex $x'\in V(G')$.

For $\tilde{x}\in V(R_G(H_1, H_2))$, let $(A_0,\ldots, A_{\ell}-1)$ be the \emph{cyclic-pattern} around $\tilde{x}$ where each $A_i$ is a subset of $\{H_1,H_2\}$, i.e.,
if $\tilde{z_0},\ldots, \tilde{z}_{\ell-1}$ are the neighbors of $\tilde{x}$ in cyclic order, then $A_i$ corresponds
to the subset of $\{H_1, H_2\}$ containing $\tilde{z_i}$. We use an identical notation for $R_{G'}(H'_1, H'_2)$.

Consider an $x'\in V(G')$ and its corresponding vertex in $R_{G'}(H'_1, H'_2)$.
First, suppose $v\not\in V(H_1)\cap V(H_2)$. If $x'\not\in\{u_0,\ldots, u_{k-1}\}$, then the cyclic-pattern
at $\tilde{x'}$ in $R_{G'}(H'_1, H'_2)$ 
is isomorphic to the cyclic-pattern at $\tilde{x}$ (replacing $H'_j$ by $H_j$, $j=1,2$), where $x$ is
the vertex in $G$ corresponding to $x'$. 
Since $H_1$ and $H_2$ are cross-free at $x$, so are $H'_1$ and $H'_2$ at $x'$. Since the subgraphs at $u_i$ are a subset
of the subgraphs at $v$, it implies that no $u_i$ is contained in both $H_1$ and $H_2$. Hence, $H_1$ and $H_2$
are cross-free at each $u_i$.

Now, if $v\in V(H_1)\cap V(H_2)$, then for any vertex $y'$ in $G'$ so that its copy $y$ in $G$ is not connected
to $v$ via a path that lies in $V(H_1)\cap V(H_2)$, the cyclic-pattern at $\tilde{y'}$ and $\tilde{y}$ are isomorphic.
If $y$ lies on a path in $V(H_1)\cap V(H_2)$ to $v$, then $\tilde{y}=\tilde{v}$ in $R_G(H_1, H_2)$. In this case, observe that the cyclic pattern at $\tilde{y'}$ is a sub-sequence of the cyclic pattern at $\tilde{y}$ since $\tilde{y'}=\tilde{u_i}$ for some $u_i\in H'_1\cap H'_2$.
Since $H_1$ and $H_2$ are non-crossing at $v$, it implies that $H'_1$ and $H'_2$ are non-crossing at $y'$. 

Since $H'_1$ and $H'_2$ are non-crossing at all vertices of $G'$, and $H_1,H_2\in\HH$ were arbitrary, it implies
$(G',\mathcal{H}')$ is cross-free.
\end{proof}

\subsection{Non-blocking Chords in $abab$-free Hypergraphs}
\label{sec:nonblocking}
In this section, we prove Lemma~\ref{lem:nblock}. Let $C=(x_0,\ldots, x_{k-1},x_0)$ be a cycle embedded in the plane 
with vertices labeled in clockwise order and
let $\mathcal{K}$ be a collection of $abab$-free subgraphs
on $C$. Since the subgraphs in $\mathcal{K}$ may not
be connected, each $K\in\mathcal{K}$ induces a collection
of \emph{runs} on $C$, where a run is a consecutive sequence of vertices on $C$ that are contained in $K$.

For $i,j\in\{0,\ldots, k-1\}$, let $[x_i, x_j]$ denote the vertices on the arc from $i$ to $j$ in clockwise order, 
with indices taken $\mathrm{mod}~k$.
Similarly, we use  $(x_i,x_j)$ to denote the open arc, i.e., consisting of the vertices on the arc from $i$ to $j$
except $x_i$ and $x_j$.
We use $(x_i,x_j]$ to denote the semi-open arc consisting of sequence of vertices from $i$ to $j$, but excluding the vertex $i$.

The addition of a chord $d=\{x_i, x_j\}$ divides $C$ into two open arcs - $(x_i,x_j)$ and $(x_j, x_i)$.
The chord $d$ \emph{blocks} a subgraph $K\in\mathcal{K}$ 
if both open arcs contain a run of $K$, and neither
$x_i$ nor $x_j$ are in $K$. 
Such a chord $d$ is called a \emph{blocking chord}. 
If $d$ does not block any subgraph in $\mathcal{K}$, it is called a \emph{non-blocking} chord. 
We show in Lemma~\ref{lem:keylem2} that there always exists a non-blocking chord $d$ 
that connects two disjoint runs of some subgraph $K\in\mathcal{K}$. 

\begin{lemma}
\label{lem:keylem2}
Let $C$ be a cycle embedded in the plane, and let $\mathcal{K}$ be a collection of $abab$-free subgraphs in the embedding of $C$.
Then, for some disconnected $K\in\mathcal{K}$, there exists a non-blocking chord joining two disjoint runs of $K$. Further, such a chord can
be computed in time $O(mn^3)$ where $m=|\KK|$ and $n=|C|$.
\end{lemma}
\begin{proof}
Assume wlog that each subgraph $K\in\mathcal{K}$ induces at least two runs in $C$,
and no two subgraphs contain the same subset of vertices of $C$.
Define a partial order $\prec_C$ on $\mathcal{K}$, where for $K, K'\in\mathcal{K}$, 
$K\prec_C K'$ iff $K\subset K'$. 
Let $K_0\in \mathcal{K}$ be a minimal subgraph with respect to the order $\prec_C$.

Let $K_0^0, \ldots, K_0^q$ denote the runs of $K_0$. We let $A$ denote the run $K_0^0$ and
let $B = \cup_{i=1}^q K_0^i$. For ease of exposition, we assume $C$ is drawn such that 
$A$ lies in the lower semi-circle of $C$ and that $B$ lies in the upper semi-circle of $C$, where the runs $K_0^1,\ldots, K_0^q$ appear in counter-clockwise order. Let $a_0,\ldots, a_r$ denote the
vertices of $A$ in clockwise order and let $b_0,\ldots, b_s$ denote the vertices of $B$ in 
counter-clockwise order. See Fig.~\ref{fig:nbchord} (a).

We show that there is a chord $d$ from a vertex in $A$ to a vertex in $B$ that
is non-blocking. In order to do so, we start with the chord $d_0 = a_0b_0$, and construct
a sequence of chords until we either find a non-blocking chord, or we end up with the chord $d_k = a_rb_s$,
which will turn out to be non-blocking. 
Having constructed chords $d_0,\ldots, d_{i-1}$, where $d_{i-1} = a_{\ell}b_{j}$, $d_i$ 
will be either the chord $a_{\ell}b_{j'}$ or
$a_{\ell'}b_{j}$, where $j'>j$ and $\ell' > \ell$. 

Next, we describe the construction of the chords. Each chord $d$ we construct satisfies the following invariant:
If $K$ is a subgraph blocked by a chord $d=a_{\ell}b_{j}$, then 
\begin{enumerate}
\item[$(i)$] The vertices of $K$ are contained in the vertices of $K_0$ in the arc $(a_{\ell},b_{j})$, and
\item[$(ii)$]
There is a vertex $k\in K\setminus K_0$ in the arc $(b_{j}, a_{\ell})$.
\end{enumerate}

Let $d_0$ denote the chord $a_0b_0$. If $d_0$ is non-blocking, we are done. Otherwise,
if any $K_1\in \mathcal{K}$ is blocked by $d_0$, there is a vertex $k\in K_1$ that lies in $(b_0,a_0)$. Since
$(b_0,a_0)$ does not contain a vertex of $K_0$, this implies $k\in K_1\setminus K_0$, and hence $d_0$
satisfies condition $(ii)$ of the invariant.
Since we assumed the subgraphs $\mathcal{K}$ are $abab$-free, 
this implies that any vertex of $K_1$ in arc $(a_0,b_0)$ is contained in $K_0$. 
This ensures that condition $(i)$ of the invariant is satisfied by $d_0$.

Having constructed $d_0=a_0b_0,\ldots, d_{i-1}=a_{\ell}b_j$, each of which satisfy conditions $(i)$ and $(ii)$
of the invariant, 
we construct $d_i$ as follows: We simultaneously scan the vertices of $B$ in counter-clockwise order from $b_j$, 
and the vertices of $A$ in clockwise order from $a_{\ell}$ 
until we find the
first vertex $x$ that belongs to a subgraph blocked by $d_{i-1}$. Let $K_i$ denote this subgraph.
If $x=b_{j'}\in B$, we set
$d_i = a_{\ell}b_{j'}$. Otherwise, $x=a_{\ell'}\in A$, and we set $d_i = a_{\ell'}b_j$. Assume without loss
of generality that $d_i = a_{\ell}b_{j'}$ (the other case is similar).

\begin{figure}[ht!]
\begin{minipage}[t]{0.45\textwidth}
\begin{center}
\includegraphics[scale=.5]{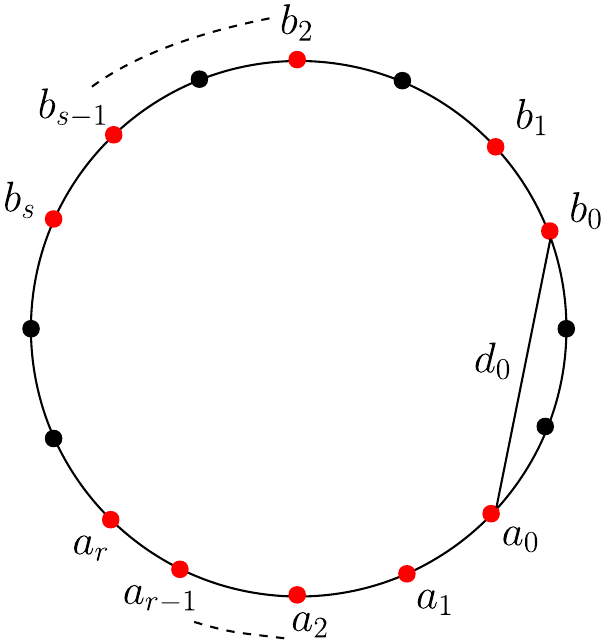}\\
(a) Ordering the vertices of $K_0$ in sets $A$ and $B$.
Here, $A=\{a_0,\ldots,a_r\}$ and $B=\{b_0,\ldots,b_s\}$
\end{center}
\end{minipage}
\hfill
\begin{minipage}[t]{0.45\textwidth}
\begin{center}
\includegraphics[scale=.5]{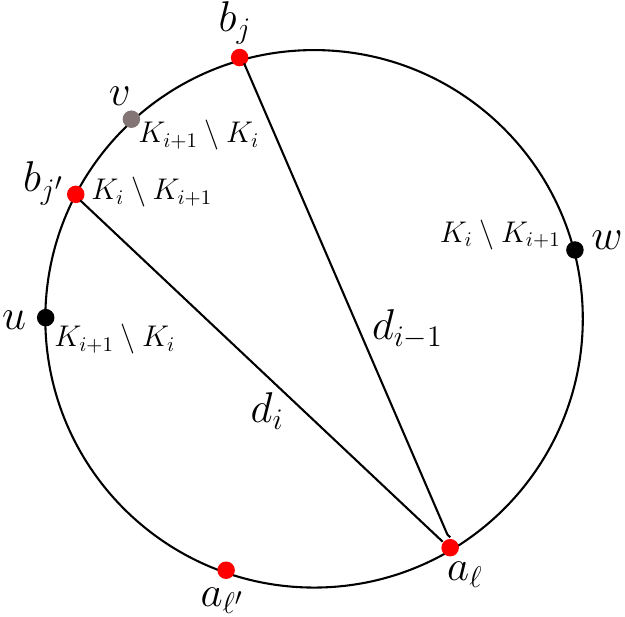}\\
(b) Adding chords between $A$ and $B$.
If $K_{i+1}\setminus K_0\neq \emptyset$ in $(a_\ell,b_{j'})$, vertices $u,b_{j'},v,w$ form $abab$ in $K_{i+1}$ and $K_i$
\end{center}
 \end{minipage}\vspace{0.5cm}
 \caption{Finding a non-blocking chord to join two disjoint runs of $K_0$.}\label{fig:nbchord}
 \end{figure}

If $d_i$ is a non-blocking chord, we are done. Otherwise, let $K_{i+1}$ denote a subgraph blocked by $d_i$. Then,
both the arcs $(a_{\ell}, b_{j'})$ and $(b_{j'},a_{\ell})$ contain a run of $K_{i+1}$, and $a_{\ell},b_{j'}\in K_0\setminus K_{i+1}$.
We now show that $d_i$ satisfies the invariant. Most of the work will go in showing that $d_i$ satisfies condition
$(i)$ of the invariant. We show this by contradiction - If $d_i$ does not satisfy invariant $(i)$, we will
exhibit a pair of subgraphs violating the $abab$-free property.

Suppose $d_i$ does not satisfy condition $(i)$ of the invariant, that is, there is a vertex $u\in K_{i+1}\setminus K_0$ 
that lies in $(a_{\ell}, b_{j'})$. 
Since $d_{i-1}$ satisfies both the conditions of the invariant, the subgraph
$K_i$ blocked by $d_{i-1}$ is contained in $K_0$ in $(a_{\ell}, b_j)$. Since $(a_{\ell}, b_{j'})\subset (a_{\ell}, b_j)$, it implies $u\not\in K_i$, and thus $u\in K_{i+1}\setminus K_i$.
By construction of $d_i$, the vertex $b_{j'}\in K_i$, and since $d_i$ blocks $K_{i+1}$, $b_{j'}\not\in K_{i+1}$. 
Thus, $b_{j'}\in K_i\setminus K_{i+1}$.

Now, we claim that $K_{i+1}$ is not blocked by $d_{i-1}$. To see this, since $d_{i-1}$ satisfies condition $(i)$ of
the invariant, for any subgraph $K'$ blocked by $d_{i-1}$, we have that $K'\subseteq K_0$ in $(a_{\ell},b_j)$.
Since $(a_{\ell}, b_{j'})\subset (a_{\ell}, b_{j})$, combined with the facts that
$u\in K_{i+1}\setminus K_0$ and that $u$ lies in $(a_{\ell}, b_{j'})$ implies that $K_{i+1}$ is not
blocked by $d_{i-1}$.
But $K_{i+1}$ is blocked by $d_i$, it follows that there is a vertex $v$ of $K_{i+1}$ in the
arc $(b_{j'},b_j]$.
Note that $v$ need not lie in $K_0$.
However, no vertex in the arc $(b_{j'},b_j]$ lies in $K_i$, since $b_{j'}$ was the first vertex encountered
that was contained in a subgraph blocked by $d_{i-1}$ when traversing the vertices of $B$ in counter-clockwise order
from $b_j$.
Therefore, $v\in K_{i+1}\setminus K_i$. 

Finally, since $d_{i-1}$ satisfies condition $(ii)$ of the invariant, it implies
that there is a vertex $w\in K_i\setminus K_0$ that lies in $(b_j, a_{\ell})$. 
Note that $u\in K_{i+1}\setminus K_0$, $a_{\ell},b_{j'}\in K_0\setminus K_{i+1}$, and $u$ lies in $(a_{\ell}, b_{j'})$.
There is no vertex $x\in K_{i+1}\setminus K_0$ that lies in $(b_{j'},a_{\ell})$, since otherwise we have an $abab$-pattern among $K_{i+1}$ and $K_0$ given by the cyclic sequence of the vertices $u,b_{j'},x,a_{\ell}$.
So, we have 
$K_{i+1}\subseteq K_0$ in
$(b_{j'}, a_{\ell})$ since the arrangement is $abab$-free. 
However, since $(b_{j}, a_{\ell})\subset (b_{j'}, a_{\ell})$ and $w\in K_i\setminus K_0$ lies in $(b_j,a_{\ell})$, it follows that
$w\not\in K_{i+1}$.
Therefore, $w\in K_i\setminus K_{i+1}$.
See Fig.~\ref{fig:nbchord} (b).

From the above arguments, it follows that the subgraphs $K_{i+1}$ and $K_i$ are not $abab$-free, 
as witnessed by the sequence
of vertices $u, b_{j'}, v$ and $w$, a contradiction.
Thus, $d_i$ satisfies condition $(i)$ of the invariant. 
The fact that $d_i$ satisfies condition $(ii)$ of the invariant follows from the fact that $K_0$ is minimal.
Otherwise, $K_{i+1}\subseteq K_0$ in $(a_{\ell}, b_{j'})$ and in $(b_{j'}, a_{\ell})$, and therefore 
$K_{i+1}\subset K_0$.

Since the set of chords is finite, the sequence of chords constructed either ends in a 
non-blocking chord, or we end up with the chord $d=a_rb_s$. We claim
that $d$ must be a non-blocking chord. Suppose $d$ blocks a subgraph $K$. Then, $(a_r, b_s)$ contains
a vertex $u\in K\cap K_0$, as $d$ satisfies invariant $(i)$ and $(ii)$. However, $(a_r, b_s)$ 
does not contain a vertex in $K_0$. 
Therefore, $d$ must be a non-blocking chord.

For any chord $d_0$ joining two disjoint runs of some $K_0$, it can be tested in $O(mn)$ if $d_0$ is a non-blocking chord since any $K\in\KK$ has at most $n$ vertices.
We scan over at most $\binom{n}{2}$ chords, and the lemma ensures one of these chords must be non-blocking. Hence, a non-blocking chord can be computed in $O(mn^3)$ time.
\end{proof}

We are now ready to prove Lemma~\ref{lem:nblock}. We do this by using Lemma~\ref{lem:keylem2} 
to add a non-blocking chord connecting two disconnected components of a subgraph, and then recursively apply Lemma~\ref{lem:keylem2}
to the two resulting cycles and their induced subgraphs.

\nblockchord*

\begin{proof}
For any subgraph $K\in\mathcal{K}$, let $n_K$ denote the number of disjoint runs of $K$ on $C$.
Let 
\begin{align*}
\mathrm{cost}(C,\mathcal{K}) = \sum_{K\in\mathcal{K}} (n_K-1) 
\end{align*}
If $\mathrm{cost}(C,\mathcal{K})=0$, then every subgraph $K\in\mathcal{K}$ consists of one run, 
and therefore $C\cap K$ is connected for each $K\in\mathcal{K}$, and
$D=\emptyset$ suffices. 

Suppose the lemma holds for all $(C',\mathcal{K}')$ with $\mathrm{cost}(C',\mathcal{K}')<N$.
Now, consider an
instance with an embedded cycle $C$ and a set $\KK$ of $abab$-free subgraphs of $C$ such that
$\mathrm{cost}(C,\mathcal{K})=N$.
By Lemma~\ref{lem:keylem2}, there is a non-blocking chord $d=\{x,y\}$ joining two disjoint runs of some subgraph $K_0\in\mathcal{K}$.

The chord $d=\{x,y\}$ divides the cycle $C$ into two arcs, $[x,y]$, and $[y,x]$. 
We construct two disjoint sub-problems on the cycles $C_{\ell}$ and $C_r$ obtained from $C$, where $C_{\ell}$ is obtained by adding the edge
$\{x,y\}$ to the arc $[y,x]$, and $C_r$ is obtained by adding the edge $\{x,y\}$ to the arc $[x,y]$.
The subgraphs in
$C_{\ell}$ and $C_r$ are respectively those induced by $\mathcal{K}$, namely 
$\mathcal{K}_{\ell} =\{K\cap C_{\ell}: K\in\mathcal{K}\}$, 
and $\mathcal{K}_r = \{K\cap C_r:K\in\mathcal{K}\}$. 
Note that $\mathcal{K}_{\ell}$ and $\mathcal{K}_r$ 
are $abab$-free on $C_{\ell}$ and $C_r$, respectively. 
Let $n^{\ell}_K$ and $n^r_{K}$ denote respectively, the
number of runs of $K\in\mathcal{K}$ in $C_{\ell}$ and in $C_r$.
Clearly, $n^{\ell}_{K_0} < n_{K_0}$ and $n^r_{K_0} < n_{K_0}$.
Also, for all other subgraphs $K'\in\mathcal{K}$, $n^{\ell}_{K'}\le n_{K'}$ and $n^r_{K'}\le n_{K'}$,
it follows that $\mathrm{cost}(C_r,\mathcal{K}_r)<\mathrm{cost}(C,\mathcal{K})$ and $\mathrm{cost}(C_{\ell},\mathcal{K}_{\ell})<\mathrm{cost}(C,\mathcal{K})$.

Hence, by the inductive hypothesis on the pair $(C_{\ell},\mathcal{K}_{\ell})$, there exists a set of non-intersecting
chords $D_{\ell}$ such that each $K\in\mathcal{K}_{\ell}$ induces a connected subgraph of $C_{\ell}\cup D_{\ell}$. Similarly, there exists a set of non-intersecting chords $D_r$ 
such that each $K\in\mathcal{K}_r$, induces a connected subgraph of $C_{r}\cup D_{r}$.
Let $D=D_{\ell}\cup D_r\cup d$.
If there is a subgraph $K\in\KK$ such that $K$ induces connected subgraphs on both $C_{r}\cup D_{r}$ and $C_{\ell}\cup D_{\ell}$ but $K$ is not connected on $C\cup D$, then $K\cap \{x,y\}=\emptyset$.
It follows that $K$ is blocked by $d$; a contradiction that $d$ is a non-blocking chord.
Hence, $D$ is a set of non-intersecting chords such that each $K\in\KK$ induces  
a connected subgraph of $(C\cup D)$.

By Lemma~\ref{lem:keylem2}, a non-blocking chord can be computed in $O(mn^3)$ time. Since $C\cup D$ gives an outerplanar graph, $|D|\le n-3$. Hence, the total running time to compute $D$ is $O(mn^4)$.
\end{proof}

\section{Construction of Supports}
\label{sec:constructiongenus}
In this section, we show that for cross-free systems on a graph of genus $g$, there exist primal, dual
and intersection supports of genus at most $g$. While the existence of an intersection support implies the
existence of the primal and dual supports,
the construction of the intersection support uses the
construction of both the primal and dual supports, and hence, we start with their description first.

If a subgraph $H\in\mathcal{H}$
contains an edge $e$, then it contains
both the end-points of $e$. Hence, $\mathcal{H}_e\subseteq\mathcal{H}_v$ for any 
$v\in V$ and $e\sim v$ where $e\sim v$ denotes that the edge $e$ is incident to $v$.
We say that a vertex $v\in V(G)$ is \emph{maximal} if $\depth(v)>\depth(e)$ for all $e\sim v$.
We need the following basic result that follows from vertex bypassing and will be used for inductive arguments for the construction of a supports.

\begin{lemma}
\label{lem:bypassdecrease}
Let $(G,\mathcal{H})$ be a cross-free graph system and 
let $v\in V(G)$ be maximal. 
Let $(G',\mathcal{H}')$ be the graph
system obtained on applying $\VB(v)$. Then, 
$\depth(u_i)<\depth(v)$ and $u_i$ is not a maximal vertex for all $u_i$ sub-dividing the edge
$\{v,v_i\}$ for $v_i\in N(v)$.
\end{lemma}
\begin{proof}
Let $e_i=\{v,v_i\}$ be the edge subdivided by the vertex $u_i$.
Since $v$ is maximal, $\HH_{e_i}\subset \HH_v$ i.e., $\depth(e_i)<\depth(v)$.
Also, note that by construction, $\depth(u_i)=\depth(e_i)$.
It follows that $\depth(u_i)<\depth(v)$ and that $u_i$ is not a maximal vertex.
Since $u_i$ was chosen arbitrarily, we have $\depth(u_i)<\depth(v)$ for all $i=1,2,\ldots,deg(v)$.
\end{proof}

\subsection{Primal Support}
\label{sec:primalsupportgenus}
Let $(G,\mathcal{H})$ be a cross-free graph system
of genus $g$ and let $c:V(G)\to\{\R,\B\}$ be any $2$-coloring of the vertices of $G$. In this section, we show that
we can construct a primal support $Q$ for $(G,\HH)$ such that the genus of $Q$ is at most $g$.

We start by showing that if no red vertex is
maximal and the subgraphs in $\mathcal{H}$ are connected (and not necessarily cross-free), then it is easy to construct a 
support.

Our construction of a support in the general case
follows by repeated application of the Vertex Bypassing operation
to a maximal red vertex. We show that each Vertex Bypassing operation 
decreases the number of maximal red vertices, and hence we
eventually obtain a system with no maximal red vertices. While the lemma below
is stated for bounded genus graphs, it works for any hereditary family closed
under edge-contractions.

Before stating the lemma, we introduce the following terminology. A pair of vertices $u,v\in V(G)$ are said to be \emph{twins} if $\mathcal{H}_u=\mathcal{H}_v$. We say that
$u$ and $v$ are \emph{adjacent twins} if in addition $u$ and $v$ are adjacent vertices in $G$.
If $c:V(G)\to\{\R,\B\}$ is a 2-coloring of $V(G)$ and $u,v\in\R(V(G))$ are twins, we say that $u$ and $v$ are \emph{red twins}.

\begin{lemma}
\label{lem:easypsupport2}
Let $(G,\mathcal{H})$ be a graph system of genus $g$ such that each $H\in\mathcal{H}$ is a connected subgraph of $G$.
Let $c:V(G)\to\{\R,\B\}$ be any 2-coloring of $V(G)$
such that no vertex in $\R(V(G))$ is maximal and there are no adjacent twins $u,v\in\R(V(G))$.
Then there is a primal support of genus at most $g$.

\end{lemma}
\begin{proof}
We prove by induction on $|\R(V(G))|$. If $|\R(V(G))|=0$, then $Q=G$ is the desired support.
Suppose the statement holds for any graph system 
satisfying the conditions in the lemma
with less than $k$ red vertices.

Let $(G,\mathcal{H})$ be a graph system with a 2-coloring $c:V(G)\to\{\R,\B\}$ satisfying the conditions of the lemma such that $|\R(V(G))|=k$.

For each vertex $u\in\R(V(G))$, we do the following:
Since $u$ is not maximal, by definition, 
there is an edge $e=\{u,v\}\in E(G)$ such that 
$\mathcal{H}_e =\mathcal{H}_u$.
We choose one such edge $e$ incident to $u$, breaking ties arbitrarily.
Since $\mathcal{H}_e\subseteq\mathcal{H}_v$, it follows
that $\mathcal{H}_u\subseteq\mathcal{H}_v$. We orient this edge from $u$ to $v$. Further, if $v\in\R(V(G))$, then $\mathcal{H}_u\subset\mathcal{H}_v$ since there are no adjacent twins in $\R(V(G))$.

It follows that the set of oriented edges induces a directed acyclic subgraph. Further, since there are no twins in $\R(V(G))$, this directed acyclic graph is an oriented forest. Since no red vertex is maximal, it implies that each tree in the forest has its edges oriented towards the root, and the root must be a blue vertex.
Choose an arbitrary tree in the forest and choose an edge $\{u,v\}$ within the tree, where $v$ is the root of the tree.
Let $G'=G/\{u,v\}$, i.e., the graph obtained by contracting
the edge $\{u,v\}$ with the new vertex taking the color
$c(v)=\B$, since $v$, the root of the tree is colored $\B$.
Let $\HH'$ be the resulting family of subgraphs, where for each $H\in\HH$, we have a subgraph $H'\in\HH'$ of $G'$ induced on the vertices $V(H)\cap V(G')$.

Note that no red vertex is maximal
in $G'$ since each red vertex in $G'$ has an outgoing edge incident on it. Further, there are no
adjacent twins in $G'$ among the red vertices as the edge contraction did not introduce new adjacencies between the red vertices. 
To see this note that the vertex obtained by
contracting the edge $\{u,v\}$ is colored $\B$, and hence, the adjacency among the red vertices
is unchanged. Since there are no adjacent red twins in $G$ it implies that there are no adjacent
twins in $G'$. Finally, note that $G'$ has genus at most $g$ since contracting an edge in $G$ 
preserves the embedding of the resulting graph on the same surface as that of $G$.

Since $G'$ has $k-1$ red vertices, by the inductive hypothesis, there is a primal support $Q$
on the blue vertices of $G'$.
The graph $Q$ is also a support for $(G,\mathcal{H})$ 
since the set of blue vertices in $H$ remains unchanged in $(G',\mathcal{H}')$ for each $H\in\mathcal{H}$.
\end{proof}

Now we prove that given a cross-free graph system $(G,\HH)$ of genus $g$ and a 2-coloring $c:V(G)\to\{\R,\B\}$, there is a primal support of genus at most $g$.
To establish the result, we repeatedly apply edge contractions for adjacent red twins, and use an induction argument on the maximum depth of a maximal red vertex via the applications of the vertex bypassing operation in combination with Lemma \ref{lem:bypassdecrease}.

\primalsupport*
\begin{proof}
If there are adjacent red twins $u$ and $v$, we contract the edge $\{u,v\}$ to a single vertex and assign it color $\R$.
This preserves the cross-freeness among subgraphs in $\HH$.
Since for any $H\in\HH$, the set of vertices in $\B(V(G))\cap V(H)$ remains unchanged, a primal support for the modified system is also a support for the given system $(G,\HH)$.
So, we assume wlog that there is no adjacent red twin.

We prove the result by induction on the maximum depth of a maximal red vertex.
Let $d$ be the maximum depth of a maximal red vertex.
Note that a non-maximal red vertex can have depth more than $d$.
For $i\in\{1,2,\ldots,d\}$, let $S_i\subseteq\R(V(G))$ be the set of maximal red vertices of depth $i$.
Let $d=1$. It follows from the definition of a maximal vertex that for each $v\in S_d$ the subgraph containing $v$ consists of a single vertex $v$.
We remove all such subgraphs from $\HH$, since they do not contain any blue vertex and thus the conditions of support are trivially satisfied for them.
Let $\HH_1=\HH\setminus S_1$.
Then, the graph system $(G,\HH_1)$ satisfies the conditions of Lemma \ref{lem:easypsupport2}, and thus admits a primal support.

Assume the statement holds for any cross-free graph system of genus $g$ with $d<k$, for some positive integer $k\ge 2$.
Let $(G,\mathcal{H})$ be a cross-free graph
system of genus $g$ and $c:V(G)\to\{\R,\B\}$ a 2-coloring such that $d=k$, i.e., the maximum depth of a 
maximal red vertex is $k$.
We apply the process below to obtain a cross-free system $(G',\mathcal{H}')$ such that the maximum depth of a maximal red vertex is at most $k-1$, and there are no adjacent twins among the red vertices. We do this in two steps.

In the first step, we apply vertex bypassing to all the vertices in $S_k$, and we assign color $\R$ to all the newly added vertices during the vertex bypassing process.
Note that if $v\in S_k$, then by Lemma \ref{lem:bypassdecrease} none of the newly added vertices is maximal.
By Lemma \ref{lem:remcrossfree}, the resulting graph system remains cross-free.
In the second step, we repeatedly contract all the adjacent red twins and color the contracted vertex $\R$.
Clearly, the resulting graph system $(G',\HH')$ remains cross-free and does not contain any adjacent red twin.

In $(G',\HH')$, we have $d\le k-1$.
To see this, note that all the vertices of $G$ in $S_k$ have been bypassed. It is possible that during the edge contractions in the second step above, two non-maximal adjacent red twins $x$ and $y$ may result in a maximal red vertex $z$. However, at least one of $x$ and $y$ must be added during the vertex bypassing operation since we assumed that $G$ does not contain adjacent red twins.
By Lemma \ref{lem:bypassdecrease}, $\depth(z)\le k-1$.

Therefore, $(G',\HH')$ is a cross-free graph system of genus $g$ such that the maximum depth of a maximal red vertex is at most $k-1$, and there are no adjacent red twins.
By the induction hypothesis, $(G',\HH')$ admits a primal support $Q$ of genus at most $g$, which is also a primal support for $(G,\HH)$ since the set of blue vertices remains unchanged for each subgraph in $\HH$.
\end{proof}

\subsection{Dual Support}
\label{sec:dualsupportgenus}
In this section, we show that for a cross-free system on a graph of genus $g$, there is a dual support of genus at most $g$.
First, we show that for constructing a dual support, we can assume that there are no two $H, H'\in\mathcal{H}$ such that $V(H)\subseteq V(H')$.

\begin{proposition}
\label{prop:nocontainment}
Let $(G,\mathcal{H})$ be a graph system.
Let $\mathcal{H}'\subseteq\mathcal{H}$ be maximal such that 
$\forall\; H, H'\in\mathcal{H}'$, $H\setminus H'\neq\emptyset$, and that for each $H\in\mathcal{H}\setminus\mathcal{H}'$, $H\subseteq H'$ for some $H'\in\mathcal{H}'$.
If $Q'$ is a dual support for 
$(G,\mathcal{H}')$ then, there is a dual support $Q^*$ for $(G,\mathcal{H})$ such that if $Q'$ has genus $g$, then $Q^*$ has genus $g$.
\end{proposition}
\begin{proof} 
Consider the containment order $\mathcal{P}=(\mathcal{H},\preceq)$ where for $H,H'\in\mathcal{H}$, $H\preceq H'$ iff $V(H)\subseteq V(H')$.
Then $\mathcal{P}$ is a partial ordered set.
Let $\ell(\mathcal{P})$ be the 
maximum length of a chain in $\mathcal{P}$, and $c(\mathcal{P})$ be the number of chains of length $\ell(\mathcal{P})$.
We prove by induction on $(\ell(\mathcal{P}),c(\mathcal{P}))$, 
that a dual support $Q'$ for a lexicographically smaller system can be extended to
a dual support $Q^*$ of a larger system so that $Q^*$ has the same genus as $Q'$.
If we show this, then we can remove all containments,
construct a support, and extend it to obtain a dual support for the original system without
increasing the genus.

If $\ell(\mathcal{P})=1$, then the elements in $\mathcal{H}$ are pairwise incomparable.
Therefore, $Q'=Q^*$ is a dual support for $(G,\mathcal{H})$.

Suppose for any $(G'',\mathcal{H}'')$, and containment order $\mathcal{P}''=(\mathcal{H}'',\preceq)$, such that 
$(\ell(\mathcal{P}''),c(\mathcal{P}''))$ is lexicographically smaller than the corresponding pair for the  containment
order $\mathcal{P} = (\mathcal{H},\preceq)$ on $(G,\mathcal{H})$, the statement of the proposition holds.

Consider a longest chain $C$ in $\mathcal{P}$, and let $H$ be the minimum element in $C$.
Let $\mathcal{H}'=\mathcal{H}\setminus\{H\}$. Then, for containment order $\mathcal{P}'$ on $\mathcal{H}'$,
$(\ell(\mathcal{P}'),c(\mathcal{P}'))$ is lexicographically smaller than $(\ell(\mathcal{P}),c(\mathcal{P}))$.
Therefore, by the inductive hypothesis, the statement of the proposition holds for $(G,\mathcal{H}')$, and let
$Q'$ denote its dual support. To obtain a dual support $Q^*$ for $(G,\mathcal{H})$, we add a vertex corresponding 
to $H$. Let $H'$ be an immediate successor of $H$ in $\mathcal{P}$ arbitrarily chosen. Connect $H$ to $H'$.
Since we added a new vertex of degree 1, it follows that if $Q'$ has genus $g$, then $Q^*$ has genus $g$.

To show that $Q^*$ is the desired dual support, let $v$ be a vertex in $G$. If $v\not\in V(H)$, then the fact that
subgraphs $\mathcal{H}_v$ correspond to a connected subgraph in $Q^*$ follows from the fact that $Q'$ is a dual
support for $(G,\mathcal{H}')$. So, consider a vertex $v\in V(H)$. Since $Q'$ is a dual support
for $(G,\mathcal{H}')$, the subgraphs $\mathcal{H}'_v$ induce a connected subgraph in $Q'$.
Since $H'$ is an immediate successor of $H$ in $\mathcal{P}$, $V(H)\subseteq V(H')$, i.e., $\mathcal{H}_v=\mathcal{H}'_v\cup\{H\}$. 
Since $H$ is adjacent to $H'$ in $Q^*$ it follows that the subgraphs $\mathcal{H}_v$ induce a connected subgraph in $Q^*$. Hence, $Q^*$ is the desired dual support for $(G,\mathcal{H)}$.
\end{proof}

We start with a special case of the problem where it is easy to obtain a dual support, 
and we obtain a support for the general instance by reducing it to this special case by repeatedly
applying Vertex Bypassing to a vertex of maximum depth.

An edge $e=\{u,v\}\in E(G)$ is said to be a \emph{special edge} if $\mathcal{H}_e=\emptyset$, but both
$\mathcal{H}_u\neq\emptyset$ and $\mathcal{H}_v\neq\emptyset$.
Since $\mathcal{H}$ is a collection of induced subgraphs
of $G$, an equivalent condition for an edge $e=\{u,v\}$ to be a special edge is that $\mathcal{H}_u\neq\emptyset$, $\mathcal{H}_v\neq\emptyset$, and  $\mathcal{H}_u\cap\mathcal{H}_v=\emptyset$.
Let $\Spl_{\mathcal{H}}(E)$ be the set of special edges in $E(G)$.
A dual support $Q^*$ for  $(G,\mathcal{H})$ satisfies the \emph{special edge property} if 
for each $e=\{u,v\}\in \Spl_{\mathcal{H}}(E)$, there is an edge in $Q^*$ between some $H\in\mathcal{H}_u$ and $H'\in\mathcal{H}_v$.

\begin{lemma}
\label{lem:dualembeddedlem}
Let $(G,\mathcal{H})$ be a graph system of genus $g$ such that  $\depth(v)\le 1$ for each $v\in V(G)$. Then, there is a dual support $Q^*$ of genus $g$ on $\mathcal{H}$ with the special edge property.
\end{lemma}
\begin{proof}
Each vertex of $G$ has depth at most 1 and therefore, no
two subgraphs in $\mathcal{H}$ share a vertex. We repeatedly
contract each edge $e=\{u,v\}$ such that $\mathcal{H}_u\subseteq\mathcal{H}_v$ until no such edge remains, breaking ties arbitrarily.
Note that this allow to contract edges incident to vertex that is not contained in any subgraph in $\HH$.
Remove multi edges if any, from the graph thus obtained.
The resulting graph $Q^*$ has genus $g$.
We show that $Q^*$ is a desired support.

For each vertex $v\in V(G)$, the
subgraphs containing $v$ trivially induce a connected subgraph
in $Q^*$ since $\depth(v)\le 1$ for all $v\in V(G)$. Consider a special edge $e=\{u,v\}\in E(G)$.
Since each of the two endpoints
belong to a subgraph in $\mathcal{H}$, and $\mathcal{H}_u\cap\mathcal{H}_v=\emptyset$, it implies that
$\mathcal{H}_u\not\subseteq\mathcal{H}_v$ and $\mathcal{H}_v\not\subseteq\mathcal{H}_u$. In the construction
of $Q^*$ therefore, $e$ is not contracted. 
Further, $u$ and $v$ are contracted to vertices
corresponding to the unique subgraph containing $u$ and $v$,
respectively. Therefore, $Q^*$ satisfies the special edge property.
\end{proof}

Now we are ready to prove the main result of this section.

\dualembedded*
\begin{proof}
By Proposition~\ref{prop:nocontainment}, 
we can assume that there are no containments in $\mathcal{H}$. 
Since $(G,\mathcal{H})$ is cross-free,
there exists a cross-free embedding of $G$ on a surface of genus $g$. Consider such an embedding of $G$. We abuse notation
 and also refer to the embedded graph by $G$.

For the graph system $(G,\mathcal{H})$, let $d$ be the maximum
depth of a vertex, and let $n_d$ denote the number of vertices in $G$ of depth $d$. The tuple $(d,n_d)$ induces a lexicographic
order:
$(d,n_d)\prec(d',n_{d'})$ if $d<d'$ or $d=d'$ and $n_d<n_{d'}$.

We prove by induction on the lexicographic order $\prec$.
However, we require a stronger inductive hypothesis that
there is a dual support with the special edge property.

If $d = 1$, then Lemma~\ref{lem:dualembeddedlem} guarantees a dual support satisfying the special edge property.
So suppose $d > 1$. Let $v$ be a vertex of maximum depth in $G$. 
We can assume that $v$ is maximal, i.e., $\mathcal{H}_e\subset\mathcal{H}_v$ for all $e\sim v$. 
Otherwise, if $\mathcal{H}_e=\mathcal{H}_v$ for some edge $e=\{u,v\}$, since $\mathcal{H}_e\subseteq\mathcal{H}_u$, it follows
that $\mathcal{H}_v\subseteq\mathcal{H}_u$. Since
$v$ has maximum depth, it implies $\mathcal{H}_v=\mathcal{H}_u$.
Contracting $e$, we obtain a lexicographically smaller system $(G/e, \mathcal{H})$. Observe that $(G/e,\mathcal{H})$ remains
cross-free since both end-points of the contracted edge $e$
are contained in the same set of subgraphs of $\mathcal{H}$.
By the inductive hypothesis, there is a dual support $Q^*$ with the special edge property 
for $(G/e, \mathcal{H})$. $Q^*$ is also a dual-support for $(G,\mathcal{H})$ since $\mathcal{H}_u=\mathcal{H}_v$. Moreover, $Q^*$ satisfies the special edge property
since the contracted edge $e$ is not a special edge, and
all special edges in $G$ survive in $G/e$.

Therefore, we can assume that for a maximum depth vertex $v$ and each $e\sim v$, $\mathcal{H}_e\subset \mathcal{H}_v$.
We apply \VB$(v)$ to obtain the system $(G',\mathcal{H}')$. By Lemma
~\ref{lem:bypassdecrease}, the new vertices $u_0,\ldots, u_{deg(v)-1}$ 
in $G'$ obtained on applying Vertex Bypassing to $v$ 
have depth at most $d-1$.
Hence, 
$(d',n_{d'}) \prec (d,n_d)$, where $d'$ and $n_{d'}$ are respectively, the
depth of a maximum depth vertex, and the number of vertices of maximum depth in $(G',\mathcal{H}')$.
We can assume that $\mathcal{H}'_{u_i}\neq\emptyset$ for each $i\in\{0,\ldots, \deg(v)-1\}$. Otherwise,
we can contract the edge $\{u_i, u_{i+1}\}$ (with indices taken~$\mathrm{mod}$ $\deg(v)$) as this does not violate the cross-free 
condition on $\mathcal{H}'$. 
Further, there is an injective correspondence between the special edges in $G$ and the special edges in $G'$:
For $i=0,\ldots, deg(v)-1$, if $\{v,v_i\}$ in $G$
is a special edge, the corresponding edge $\{u_i, v_i\}$ 
is a special edge in $G'$.

By Lemma~\ref{lem:keylem}, each subgraph $H\in\mathcal{H}'$ is connected in $G'$,
and by Lemma~\ref{lem:remcrossfree}, $(G',\mathcal{H}')$ is cross-free.
By the inductive hypothesis, there is a dual support $Q^*$ for $(G',\mathcal{H}')$ satisfying the special edge property. 
We show that $Q^*$ is also a support for $(G,\mathcal{H})$. 

For each $u\neq v\in V(G)$, it follows from the inductive hypothesis that $\mathcal{H}_u$ induces a connected subgraph of $Q^*$.
It remains to show that $\mathcal{H}_v$ induces a connected subgraph of $Q^*$.
Let $C$ denote the cycle $(u_0,\ldots, u_{\deg(v)-1},u_0)$ added in \VB$(v)$.
Since we assumed (by Proposition~\ref{prop:nocontainment}) that $\mathcal{H}$ has no containments, it follows that
$\cup_{i=0}^{\deg(v)-1} \mathcal{H}'_{u_i} = \mathcal{H}_v$ as there is no subgraph containing only the vertex $v$.
If none of the edges of $C$ are in $\Spl_{\mathcal{H}'}(E)$, then $\mathcal{H}_v$ is connected since adjacent vertices of $C$ 
share at least one subgraph and by our assumption on applying $\VB(v)$, $\mathcal{H}'_{u_i}\neq\emptyset$ for any $i=\{0,\ldots, deg(v)-1\}$. 
On the other hand, if an edge $e=\{u_i, u_{i+1}\}$ (where indices are taken $\mathrm{mod} \deg(v)$) of $C$ is in
$\Spl_{\mathcal{H}'}(E)$, by the inductive hypothesis, since
$Q^*$ satisfies the special edge property, at least one subgraph from $\mathcal{H}'_{u_i}$ and one subgraph from $\mathcal{H}'_{u_{i+1}}$ are adjacent in $Q^*$.
Since $\cup_{i=0}^{\deg(v)-1}\mathcal{H}'_{u_i}=\mathcal{H}_v$, and $C$ is a cycle,
it follows that $\mathcal{H}_v$ is connected and thus $Q^*$, is the desired dual support for $(G,\mathcal{H})$ that also satisfies the special edge property.
\end{proof}

By the special edge property of a dual support, we make use of the following observation for the construction of an intersection support in the following section.

\begin{observation}\label{obs:depth1adjacency}
Given a graph system $(G,\HH)$, let $Q^*$ be a dual support with the special edge property.
Let $H_u=\{u\}$ and $H_v=\{v\}$ be subgraphs in $\HH$ for some depth 1 vertices $u,v\in V(G)$.
Then we can assume wlog that $u$ and $v$ are adjacent in $G$ if and only if $H_u$ and $H_v$ are adjacent in $Q^*$.
Indeed, if $u,v$ are adjacent in $G$, then $H_u,H_v$ should be adjacent in $Q^*$ by the special edge property. If $u,v$ are not adjacent in $G$ but $H_u,H_v$ are adjacent in $Q^*$, then removing the edge $\{H_u,H_v\}$ from $Q^*$ does not violate the conditions of a dual support with the special edge property.
\end{observation}

\subsection{Intersection Support}
\label{sec:intsupportgenus}
In this section, we show how we can construct an intersection support for a cross-free intersection system $(G,\mathcal{H},\mathcal{K})$.
A vertex $v\in V(G)$ s.t. $\mathcal{H}_v=\emptyset$, but $\mathcal{K}_v\neq\emptyset$ is called a $\mathcal{K}$-vertex, 
i.e., a vertex of $G$ that is contained in one or more subgraphs in $\mathcal{K}$ but none of the subgraphs in $\mathcal{H}$.

Our proof proceeds in three steps: First, we show that if there are no $\mathcal{K}$-vertices in $G$, then a dual support
for the graph system $(G,\mathcal{H})$ is also an intersection support for $(G,\mathcal{H},\mathcal{K})$.
Secondly, if there are $\mathcal{K}$-vertices in $G$, then we apply vertex bypassing so that none of the $\mathcal{K}$-vertices
are maximal w.r.t. $\KK$. We can then add a set of \emph{dummy subgraphs} $\mathcal{F}$ so that there are no $\mathcal{K}$-vertices
in $\mathcal{H}\cup\mathcal{F}$. Now, a dual support $Q^*$ for $(G,\mathcal{H}\cup\mathcal{F})$ is an intersection support
for $(G,\mathcal{H}\cup\mathcal{F},\mathcal{K})$. Finally, since $Q^*$ is an intersection support, 
each $K\in\mathcal{K}$ induces a connected subgraph $\HH_K$ of $Q^*$ (though not necessarily cross-free). However, since no 
$\mathcal{K}$-vertex was maximal, we can show that no $\mathcal{F}$ vertex in $Q^*$ is maximal. 
Hence, we can color the vertices of $Q^*$: vertices corresponding to $\mathcal{F}$ get colored $\R$, and
those corresponding to $\mathcal{H}$ get colored $\B$. By appealing to Lemma \ref{lem:easypsupport2},
we obtain the desired intersection support $\tilde{Q}$.

In the lemma below, observe that we only require
that the subgraphs in $\mathcal{K}$ are connected and
not necessarily cross-free.

\begin{lemma}
\label{lem:conn}
Let $(G,\mathcal{H})$ be a cross-free system of genus $g$ and let $\mathcal{K}$ be a set of connected subgraphs of $G$.
If $G$ does not contain a $\mathcal{K}$-vertex, 
then a dual support 
$Q^*$ for $(G,\mathcal{H})$ of genus $g$ satisfying the special edge property,
is also an intersection support for $(G,\mathcal{H},\mathcal{K})$.
\end{lemma}
\begin{proof}
By Theorem~\ref{thm:dualembedded}, there is a dual support $Q^*$ of genus at most $g$ for the graph system $(G,\mathcal{H})$ with the
special edge property.
That is, for any vertex $v\in V(G)$, the subgraphs $\mathcal{H}_v$ induce a connected subgraph of $Q^*$, and 
for any edge $\{u,v\}\in E(G)$ such that $\mathcal{H}_u\neq\emptyset$, $\mathcal{H}_v\neq\emptyset$, and
$\mathcal{H}_{\{u,v\}}=\emptyset$, there is a subgraph $H\in\mathcal{H}_u$ and a subgraph $H'\in\mathcal{H}_v$
such that $H$ and $H'$ are adjacent in $Q^*$.

For any $K\in\mathcal{K}$, let $H, H'\in\mathcal{H}_K$.
We show that there is a path in $Q^*$ between $H$ and $H'$ such that each vertex of this path corresponds to a subgraph in $\mathcal{H}_{K}$.
Let $u\in H\cap K$ and $v\in H'\cap K$. 

Since $K$ is connected, there is a path $P=(u=u_0, u_1,\ldots, u_k=v)$
that lies in $K$.
By assumption, none of the vertices of $P$ is a $\mathcal{K}$-vertex. Since $Q^*$ is a dual support for $(G,\mathcal{H})$,
for any $i\in\{0,\ldots, k\}$, the subgraphs in $\mathcal{H}_{u_i}$ induce a connected subgraph in
$Q^*$. For any edge $\{u_i, u_{i+1}\}$, $i=0,\ldots, k$, either there is a subgraph $H\in\mathcal{H}$
that contains the edge $\{u_i,u_{i+1}\}$, or by the special edge property, there is a subgraph
in $\mathcal{H}_{u_i}$ that is adjacent to a subgraph in $\mathcal{H}_{u_{i+1}}$ in $Q^*$. 
Since $\mathcal{H}_{u_i}\subseteq \mathcal{H}_K$ for each
$i=0,\ldots, k$, there is a path
in $Q^*$ between $H$ and $H'$ consisting only of subgraphs in $\mathcal{H}_K$. 
Since $K$ was chosen arbitrarily, $Q^*$ is an intersection support for $(G,\mathcal{H},\mathcal{K})$.
\end{proof}

Now, we apply vertex bypassing to obtain a graph system where none of the $\mathcal{K}$-vertices are maximal.
For a $\mathcal{K}$-vertex $v$, if an edge $e\sim v$ is such that $\mathcal{K}_v=\mathcal{K}_e$,
we say that $e$ is \emph{full} for $v$. If a $\mathcal{K}$-vertex does not have a full
edge incident on it, then we say that it is maximal w.r.t. $\KK$ subgraphs. In this case, 
$\mathcal{K}_e \subset \mathcal{K}_v $ for all ${e}\sim v$.
Note that a maximum depth $\mathcal{K}$-vertex need not be maximal.
In the following, we repeatedly apply vertex bypassing to a maximal $\mathcal{K}$-vertex of
maximum depth until no $\mathcal{K}$-vertex is maximal. However, we require an additional 
property like in the construction of the primal support, that there are no adjacent twins. 
We recall the definition of twins and adjacent twins here. 
If two $\KK$-vertices $u,v$ of $G$ are contained in the same set of subgraphs of $\KK$, 
they are said to be \emph{twins}. If, in addition, $u$ and $v$ are adjacent, then they are said to
be adjacent twins.

The proof of the lemma below follows from arguments similar to those of Theorem \ref{thm:primalsupport}.

\begin{lemma}
\label{lem:removeMaximal}
Let $(G,\mathcal{H},\mathcal{K})$ be a cross-free intersection system of genus $g$. 
Then, we can modify the arrangement 
to a cross-free arrangement $(G',\mathcal{H},\mathcal{K}')$ such that (i) $G'$ has genus $g$, (ii) no $\mathcal{K}$-vertex in $(G',\HH,\KK')$ is maximal, (iii) no two $\KK$-vertices are adjacent twins, and (iv)
an intersection support for $(G',\mathcal{H}, \mathcal{K}')$ is also an intersection support for $(G,\mathcal{H}, \mathcal{K})$.
\end{lemma}
\begin{proof}
We assume that a cross-free embedding of $(G,\mathcal{H},\mathcal{K})$ is given, and with a slight abuse of notation, we use $G$ to also refer to the embedded graph.
We will modify the cross-free graph system $(G,\HH,\KK)$ to a cross-free graph system $(G',\HH,\KK')$ such that $G'$ has an embedding on the same surface as $G$ does, and that the intersection hypergraph defined by $(G,\HH,\KK)$ is isomorphic to that defined by $(G',\HH,\KK')$.
This will prove $(i)$ and $(iv)$.
Also, in the process of obtaining $(G',\HH,\KK')$, as long as there are adjacent twins among $\KK$-vertices, we contract them to a single vertex.
Thus, we will be done with condition $(iii)$ as well. So, in the rest of the proof, we claim to establish condition $(ii)$ of the lemma.

Let $d$ denote the maximum depth of a maximal $\KK$-vertex in $G$. We prove the argument by induction on $d$.
For $i=1,2,\ldots, d$, let $S_i$ denote the set of maximal $\KK$-vertices of depth $i$.
Suppose $d=1$, and consider any $v\in S_1$.
By the definition of a maximal vertex w.r.t. $\KK$, the subgraph $K\in\KK$ containing $v$, does not contain any other vertex of $G$. 
Thus $\HH_{K}=\emptyset$.
Let $\KK_1\subseteq\KK$ denote the collection of all subgraphs consisting of single vertices in $S_1$, and let $\KK'=\KK\setminus \KK_1$.
Then $(G',\HH,\KK')$ is the required graph system, where $G'=G$.

Let our claim be true for $d<k$ for some positive integer $k\ge 2$.
Consider a cross-free graph system $(G,\HH,\KK)$ and a 2-coloring $c:V(G)\to\{\R,\B\}$ such that $d=k$.
We apply vertex bypassing to all the vertices in $S_k$.
Note that this does not modify any subgraph in $\HH$.
For any pair of adjacent twins created among the $\KK$-vertices, we contract them to a single $\KK$-vertex until there is no such twin.
Let $(G'',\HH,\KK'')$ denote the resulting graph system.
By Lemma \ref{lem:remcrossfree} and the contraction among the adjacent $\KK$-twins, it follows that $(G'',\KK'')$ remains cross-free.
Since for any vertex $v\in V(G)$ that lies in a subgraph $H\in\HH$, the cyclic order of its neighbours is preserved in $G''$, the graph system $(G'',\HH)$ is cross-free.
Therefore, $(G'',\HH,\KK'')$ is a cross-free intersection system.
By the construction of $G''$, note that a subgraph $H\in\HH$ intersects a subgraph $K\in\KK$ in $G$ if and only if $H$ intersects in $G''$ with the representative $K''$ of $K$.
It follows that the intersection hypergraphs defined by the two intersection systems are isomorphic.

By an argument similar to that in the proof of 
Theorem \ref{thm:primalsupport}, the maximum depth of a maximal $\KK$-vertex in $(G'',\HH,\KK'')$ is at most $k-1$.
By the induction hypothesis, $(G'',\HH,\KK'')$ can be modified to a graph system $(G',\HH,\KK')$
satisfying properties $(i)$-$(iv)$ of the lemma.
$(G',\HH,\KK')$ is the required cross-free intersection system since the intersection hypergraph defined by $(G',\HH,\KK')$ is isomorphic to that defined by $(G'',\HH,\KK'')$, which itself is isomorphic to the intersection hypergraph defined by $(G,\HH,\KK)$.
\end{proof}

By Lemma \ref{lem:removeMaximal}, it is sufficient to prove an intersection support for $(G,\HH,\KK)$ when no $\KK$-vertex is maximal, and there is no adjacent twin among $\KK$-vertices.
At each $\mathcal{K}$-vertex $v$, we add a \emph{dummy subgraph} $F_v$.
Let $\mathcal{F}$ denote the set of dummy subgraphs thus added. For the intersection system $(G,\mathcal{H}\cup\mathcal{F},\mathcal{K})$,
we can obtain an intersection support $Q$ by applying Lemma \ref{lem:conn}.
We now show that we can obtain an intersection support for $(G,\mathcal{H},\mathcal{K})$ from $Q$ using Lemma \ref{lem:easypsupport2} and Observation \ref{obs:depth1adjacency}.

\intsupport*
\begin{proof}
If $G$ does not contain any $\mathcal{K}$-vertex, then by Lemma \ref{lem:conn}, we are done.
So, suppose that there are some $\KK$-vertices in $G$.
By Lemma \ref{lem:removeMaximal}, we can assume without loss of generality that no $\KK$-vertex is maximal, and that there are no
adjacent twins among the $\KK$-vertices.

At each $\mathcal{K}$-vertex $v$ of $G$, we add a dummy subgraph $F_v$. Let $\mathcal{F}$ denote the set of dummy subgraphs
added.
Let $\HH'=\HH\cup\mathcal{F}$ and consider the intersection system $(G,\HH',\KK)$. By Lemma \ref{lem:conn}, there is a dual support $Q$ for $(G,\HH')$ with the special edge property such that $Q$ is an intersection support for $(G,\mathcal{H}',\mathcal{K})$, and $Q$ has genus at most $g$.
Since $Q$ is an intersection support, each $K\in\mathcal{K}$ induces a connected subgraph $\HH'_K$ of $Q$. 
Abusing notation, we use the labels of the subgraphs in $\mathcal{H}\cup\mathcal{F}$ to denote their corresponding vertices in $Q$.
Consider a dummy subgraph
$F_u$ corresponding to a $\mathcal{K}$-vertex $u$. Since $u$ is not maximal in $G$ w.r.t. $\KK$, it has a full edge $e=\{u,v\}$ incident to it,
i.e., $\mathcal{K}_u=\mathcal{K}_{e}$.
But, this implies $\mathcal{K}_u\subseteq\mathcal{K}_v$.
Note that for every edge incident to $u$ in $G$, there is a corresponding edge in the dual support $Q$ as a special edge incident to the dummy subgraph $F_u$, since $F_u$ consists of a single vertex $u$, and $|\HH'_u|=1$.
In particular, let $e'$ be the copy of $e$ incident to $F_u$ in $Q$.
If $X\in\HH\cup\mathcal{F}$ is the other end of $e'$, then $\KK_{F_u}\subseteq\KK_X$ in the graph system $(Q,\KK)$ i.e., $F_u$ is not a maximal vertex in $(Q,\KK)$.
Hence, none of the dummy subgraphs $\mathcal{F}$ corresponds to a maximal vertex in $(Q,\KK)$.

Now we assign a 2-coloring to the vertices of $Q$.
Let $c:V(Q)\to\{\R,\B\}$ be such that $c(F)=\R$ for each $F\in\mathcal{F}$, and $c(H)=\B$ for each $H\in\HH$.
To complete the argument, we need to show that there are no adjacent red
twins in the graph system $(Q,\KK)$.
Consider any two red vertices $F_u$ and $F_v$ of $Q$, corresponding to some pair of vertices $u$ and $v$ of G, respectively.
Since $|\HH_u|=1=|\HH_v|$, by Observation \ref{obs:depth1adjacency}, we can assume that $F_u$ and $F_v$ are adjacent in $Q$ if and only if $u$ and $v$ are adjacent in $G$.
If $F_u$ and $F_v$ are not adjacent, then trivially they are not adjacent red twins.
So, suppose they are adjacent.
Then $u$ and $v$ are also adjacent in $G$. By assumption, $G$ does not contain adjacent $\KK$-vertices that are twins, thus $\KK_u\ne\KK_v$.
It follows that in the graph system $(Q,\KK)$, we have $\KK_{F_u}\ne\KK_{F_v}$, and hence, $F_u$ and $F_v$ are not twins in $Q$.
Therefore, $Q$ satisfies the conditions of Lemma \ref{lem:easypsupport2} that no \emph{red} vertex is maximal and each $K\in\mathcal{K}$
induces a connected subgraph of $Q$. Therefore, the graph system $(Q,\KK)$ admits a support $\tilde{Q}$ on $\mathcal{H}$
of genus at most $g$.
Since each $K\in\mathcal{K}$ induces a connected subgraph in $\tilde{Q}$, we have the desired intersection support for $(G,\HH,\KK)$.
\end{proof}

\section{Applications}
\label{sec:applications}
In this section, we describe some applications of the existence of supports to Packing, Covering and Coloring problems. 
For many of these applications, we only require the underlying graph to satisfy conditions that are weaker than that for
a support. To that end, define weak supports and weak bipartite supports.

\smallskip\noindent
{\bf Weak (Bipartite) Supports.}
For a hypergraph $(X,\mathcal{E})$,  a \emph{weak support} is a graph $Q'$ on $X$ s.t. for each hyperedge $E\in\mathcal{E}$ with at least two vertices,
the induced subgraph of $Q'$ on the elements of $E$ is \emph{non-empty}, i.e., it has an edge between some pair of vertices in $E$.
Given a 2-coloring 
$c:X\to\{\R,\B\}$, a \emph{weak bipartite support} is a graph $Q''$ on $X$ s.t. for each hyperedge
$E\in\mathcal{E}$ with vertices of both colors, there is an edge in $Q''$ between a vertex in $E\cap\B(X)$
and a vertex in $E\cap\R(X)$.

It is easy to see that if $Q$ is a support for a hypergraph, then it is also a weak bipartite support. Indeed, for any edge $E$ with vertices of two colors,
since the induced graph is connected, there is an edge between some $E\cap B(V)$ and $E\cap R(V)$. Further, a weak bipartite support is also a weak support. 
However, it is not difficult to construct instances where a sparse weak bipartite support exists, 
but a sparse support does not.

For applications of Hypergraph coloring, we note that a weak support is sufficient, while for 
 many (though not all) of the Packing and Covering problems we consider, a weak bipartite support is sufficient.

\subsection{Algorithms for Packing and Covering Problems}
\label{sec:packcover}

We start with definitions of the problems we consider. 
The input for the problems below is a set $X$ along with a collection of subsets 
$\mathcal{S}$ of $X$.

\begin{definition}[Set Packing] Pick a \emph{largest} sub-collection $\mathcal{S}'\subseteq\mathcal{S}$
s.t. $|\{S\in\mathcal{S}': S\ni x\}|\le 1$ for all $x\in X$.
\end{definition}

\begin{definition}[Point Packing]
Pick a \emph{largest} set $Y\subseteq X$ s.t. for all $S\in\mathcal{S}$,
$|Y\cap S|\le 1$.
\end{definition}

\begin{definition}[Set Cover]
Assuming $\cup_{S\in\mathcal{S}} S = X$,
select a \emph{smallest} sub-collection $\mathcal{S}'\subseteq\mathcal{S}$ such that $\cup_{S\in\mathcal{S}'} S = X$.
\end{definition}

\begin{definition}[Hitting Set]
Select a \emph{smallest} cardinality set $Y\subseteq X$ s.t. 
$Y\cap S\neq\emptyset$, $\forall S\in\mathcal{S}$.
\end{definition}

We also define a generalized version of packing and covering problems. Let
$X$ be a set,
and let $\mathcal{S}$ and $\mathcal{T}$ be two collections of
subsets of $X$. This defines an \emph{intersection hypergraph}
$(\mathcal{S},\{\mathcal{S}_T\}_{T\in\mathcal{T}})$, where 
for $T\in\mathcal{T}$, $\mathcal{S}_T=\{S\in\mathcal{S}:S\cap T\neq\emptyset\}$.

\begin{definition}[Generalized Packing]
The Generalized Packing problem is the Set/Point Packing problem on the intersection hypergraph $(\mathcal{S},\{\mathcal{S}_T\}_{T\in\mathcal{T}})$.
\end{definition}

\begin{definition}[Generalized Covering]
The Generalized Covering problem is the Set Cover/Hitting Set problem on the intersection hypergraph $(\mathcal{S},\{\mathcal{S}_T\}_{T\in\mathcal{T}})$.
\end{definition}

We also consider capacitated variants of the packing problems.

\begin{definition}[Generalized Capacitated Packing]
Given an intersection system $(G,\mathcal{H},\mathcal{K})$ and a capacity function $\cp:\mathcal{K}\to\mathbb{N}$, the goal is to find a largest
sub-collection $\mathcal{H}'\subseteq\mathcal{H}$ s.t. for each $K\in\mathcal{K}$, $|\mathcal{H}'_K|\le\cp(K)$.
\end{definition}

In the generalized version of the problems, taking either the set $\mathcal{S}$, or the set $\mathcal{T}$ as singleton
sets of $X$, we obtain the Set Cover/Set Packing, or Point Packing/Hitting Set problems. Hence, any result obtained 
for the generalized version of the problem applies immediately to the classic versions of the problems. 

We also study some graph problems on an intersection graph defined by a hypergraph.
Let $(X,\mathcal{S})$ be a hypergraph. The \emph{intersection graph} defined by
$\mathcal{S}$ is the graph $I(\mathcal{S})=(\mathcal{S},F)$, where $\{S,S'\}\in F
\Leftrightarrow S\cap S'\neq\emptyset$.

\begin{definition}[Dominating Set] Let $(X,\mathcal{S})$ be a hypergraph. 
For the intersection graph $I(\mathcal{S})$, a set $\mathcal{D}\subseteq\mathcal{S}$ 
is a Dominating Set if for each $S\in\mathcal{S}$, either $S\in\mathcal{D}$
or $S\cap S'\neq\emptyset$ for some $S'\in\mathcal{D}$. The goal is to find a dominating
set of minimum cardinality.
\end{definition}

\begin{definition}[Independent Set] Let $(X,\mathcal{S})$ be a hypergraph. 
For the intersection graph $I(\mathcal{S})$, a set $\mathcal{K}\subseteq\mathcal{S}$ 
is an Independent Set if the sets in $\mathcal{K}$ are pairwise non-adjacent.
The goal is to find an independent set of maximum cardinality.
\end{definition}

\begin{definition}[Vertex Cover] Let $(X,\mathcal{S})$ be a hypergraph. 
For the intersection graph $I(\mathcal{S})$, a set $\mathcal{C}\subseteq\mathcal{S}$ 
is a Vertex Cover if for each edge in $I(\mathcal{S})$ at least one of its end-points is in 
$\mathcal{C}$. The goal is to find a vertex cover of minimum cardinality.
\end{definition}

In abstract set systems, there is essentially no difference between the dual pairs of problems Point Packing/Set Packing, or Hitting Set/Set Cover.
The
Set Packing problem contains
as a special case, the Independent Set problem on graphs, and due to~\cite{hastad1999clique}, the problem is therefore 
hard to approximate beyond $n^{1-\epsilon}$ for any $\epsilon>0$ unless NP$=$ZPP.
Further, the Set Cover problem admits an $O(\log n)$-approximation, see for example~\cite{chvatal1979greedy,lov1975notdef}, 
and that this is tight due to~\cite{feige1998threshold}.

In geometric settings however, a problem and its dual may be quite different. For example,~\cite{mustafa2015quasi} showed that there exists a 
QPTAS\footnote{A QPTAS is an algorithm whose approximation guarantee is the same as a PTAS, but whose running time is allowed to be a quasi-polynomial in the input, i.e., of the form $2^{polylog(n)}$.} for the problem of covering a set of points
in the plane with disks, but such a result is not known for the Hitting Set variant of the problem.

Our results, in combination with previous work, imply PTAS for a general class of geometric hypergraphs
in the plane, or on an oriented surface of bounded genus.

\subsubsection{PTAS via local-search }
\smallskip\noindent
We show that a simple local-search framework leads to PTAS for the packing and covering problems defined above.
We start with a brief description of this technique.\\
{\bf Local-search Framework:}
For a parameter $k\in\mathbb{N}$, the algorithms in the framework have the following form:
We start with an arbitrary feasible solution. While there is a feasible solution of
better value within a \emph{$k$-neighborhood} of the current solution, replace the current solution with this better solution.
When no such improvement is possible, return the current solution. The algorithms guarantee an approximation factor
of $O(1\pm 1/k)$ and have running time $O(N^{1/k^2})$, where $N$ is the input size. This basic framework is a basis for 
several PTAS for geometric packing and covering problems. 
See the following references~\cite{DBLP:conf/walcom/AschnerKMY13, ChanH12, mustafa2010improved} for concrete algorithms and analysis under this framework.

The key to proving that the local-search framework yields a 
PTAS is to show the existence of a graph satisfying the
\emph{local-search property}. Such a graph is called a
\emph{local-search graph}.

\begin{definition}[Local-Search Property]
\label{def:lsp}
For a Packing/Covering problem, let $\mathcal{L}$ denote a solution
returned by the local-search framework and let $\mathcal{O}$ be
an optimal solution. A graph $G=(\mathcal{L}\cup\mathcal{O},E)$
is said to satisfy the \emph{Local-Search Property} if it satisfies the following
\emph{local} and \emph{global} properties.

\begin{enumerate}
\item {\bf Local Property:} Below $N(.)$ refers to the neighborhood in $G$.
\begin{enumerate}
\item (Maximization) For any $\mathcal{O}'\subseteq\mathcal{O}$,
$(\mathcal{L}\cup\mathcal{O}')\setminus N(\mathcal{O}')$ is
a feasible solution.
\item (Minimization) For any $\mathcal{L}'\subseteq\mathcal{L}$,
($\mathcal{L}\setminus\mathcal{L}')\cup N(\mathcal{L}')$ is
a feasible solution.
\end{enumerate}

\item {\bf Global Property:} $G$ comes from a hereditary family of graphs that has sub-linear separators. 
\end{enumerate}
\end{definition}

The local property captures a subset of the local moves, and the global property is used to bound the approximation factor
guaranteed by the algorithm. The argument is roughly as follows.
Since the local-search graph has sub-linear separators, 
we can recursively remove separators in the graph until each
component has size at most $k$, where $k$ is the local-search
parameter. At this point, since our solution is locally optimal,
within each component, the local solution is \emph{better}
than the optimal solution restricted to this component, and
the total number of elements in the separators is small enough
that we can guarantee a PTAS. See~\cite{DBLP:conf/walcom/AschnerKMY13, ChanH12, mustafa2010improved} for a detailed analysis of the technique.

Consider the Set Packing problem for the cross-free system 
$(G,\mathcal{H})$ of genus $g$.
Let $\mathcal{L}$ denote the solution returned by local-search, and let 
$\mathcal{O}$ denote an optimal solution. Assume for simplicity that $\mathcal{L}\cap\mathcal{O}=\emptyset$.
The local property that the local-search graph on $\mathcal{L}\cup\mathcal{O}$ must satisfy is that
for any $\mathcal{O}'\subseteq\mathcal{O}$, $(\mathcal{L}\cup\mathcal{O}')\setminus N(\mathcal{O}')$
is a feasible solution. That is, for each $v\in V$, there
is an edge between some $L\in\mathcal{L}_v$ and some $O\in\mathcal{O}_v$, in other words, 
a weak bipartite support for the dual hypergraph defined by the graph system $(G,\LL\cup\OO)$ where $\mathcal{L}$ and
$\mathcal{O}$ are the two colors.
Similarly, for the Point Packing, Set Cover or
Hitting Set problems, a weak bipartite support of bounded genus yields is a local-search graph.
For the capacitated Packing/Covering problems however, we require a support in order to construct the
desired local-search graphs.

We now construct a local-search graph for the Generalized Capacitated Packing Problem for a
cross-free intersection system $(G,\mathcal{H},\mathcal{K})$ of genus $g$.
The existence of a PTAS follows the general framework
in~\cite{DBLP:journals/dcg/RamanR22}. However, the existence of a suitable local-search graph requires more work.
Let $\mathcal{L}$ and $\mathcal{O}$ be respectively, the solution returned by the local-search algorithm and
an optimal solution. We can assume that $\mathcal{L}\cap\mathcal{O}=\emptyset$. We require that for any
$\mathcal{O}'\subseteq\mathcal{O}$, $(\mathcal{L}\cup\mathcal{O}')\setminus N(\mathcal{O}')$ is a feasible solution.
The problem is that since each $K\in\mathcal{K}$ is a subgraph spanning many vertices, an $L\in\mathcal{L}_K$ and
an $O\in\mathcal{O}_K$ may not share a vertex. So, if we try to use a dual support for $(G,\mathcal{K})$, then we lose
the information that $L$ and $O$ intersect the same $K\in\mathcal{K}$. We can instead start with an intersection
support $\tilde{Q}$ for $(G,\mathcal{K},\mathcal{H})$. Then, for each $K\in\mathcal{K}$, then for any $K\in\mathcal{K}$,
an $L\in\mathcal{L}_K$ and an $O\in\mathcal{O}_K$ share a vertex in $\tilde{Q}$. We can use $\tilde{Q}$ to construct a local-search
graph $\Xi$ on $\mathcal{L}\cup\mathcal{O}$. To prove that the resulting graph has sub-linear sized separators, we can use the
fact that $\tilde{Q}$ has sub-linear sized separators. However, the size of the separator in $\tilde{Q}$ is a function
of $|\mathcal{K}|$, while for $\Xi$, we require a separator whose size is a function of $|\mathcal{L}\cup\mathcal{O}|$. 
Therefore, we first \emph{sparsify} $\mathcal{K}$ to obtain a subset $\mathcal{K}'$, and then construct an intersection support
for $(G,\mathcal{K}',\mathcal{L}\cup\mathcal{O})$. We finally use this intersection support to construct the required local-search
graph.

In particular, we do the following: first, we show that there is a set $\mathcal{K}'\subseteq\mathcal{K}$ of size
$O(|\mathcal{L}\cup\mathcal{O}|)$ that \emph{witness the intersection} of an $L\in\mathcal{L}$
and an $O\in\mathcal{O}$ with a $K\in\mathcal{K}$. This follows from the fact that there is an intersection support
of bounded genus, combined with the Clarkson-Shor technique of~\cite{clarkson1989applications, sharir2003clarkson}. We next use these witness subgraphs
to define a new intersection system and construct
an intersection support for this system. Notably, in this new intersection system, the roles of $\mathcal{H}$ and $\mathcal{K}$
are reversed. Using the intersection support thus created, we are
finally able to construct the local-search graph on $\mathcal{L}\cup\mathcal{O}$, and hence a PTAS follows via the local-search framework.
We start by showing that there is a small witness set $\mathcal{K}'$ for $\mathcal{L}\cup\mathcal{O}$.
\begin{lemma}
\label{lem:DeltaK}
Let $(G,\mathcal{H},\mathcal{K})$ be a cross-free intersection system of genus $g$ such that for all $K\in\mathcal{K}$, $2\le |\mathcal{H}_K|\le\Theta$
for some $\Theta\ge 2$. Then, there is a set $\mathcal{K}'\subseteq\mathcal{K}$ with $|\mathcal{K}'| = O(\Theta\cdot(g+|\mathcal{H}|))$
and s.t. for any pair $H,H'\in\mathcal{H}_K$ for some $K\in\mathcal{K}$, there is a $K'\in\mathcal{K}'$ s.t. $H,H'\in\mathcal{H}_{K'}$.
\end{lemma}
\begin{proof}
Let $\mathcal{K}'\subseteq\mathcal{K}$ be a minimal subset such that for any $H, H'\in\mathcal{H}_K$ for some $K\in\mathcal{K}$,
there is a $K'\in\mathcal{K}'$ s.t. $H, H'\in\mathcal{H}_{K'}$. By minimality of $\mathcal{K}'$, for any $K'\in\mathcal{K}'$, there
is a pair $H, H'\in\mathcal{H}$ that appear together only in $\mathcal{H}_{K'}$. 

By Theorem~\ref{thm:intsupport}, there is an intersection support $\tilde{Q}$ of genus at most $g$
for $(G,\mathcal{H},\mathcal{K}')$. The construction in Theorem~\ref{thm:intsupport} also produces an embedding of $\tilde{Q}$ on the
surface. We can assume wlog that $\tilde{Q}$ is simple and is $2$-cell embedded.
Since $\tilde{Q}$ has genus at most $g$, by Euler's formula $|E(\tilde{Q})|\le 3(g+|\mathcal{H}|)$.
If a pair of subgraphs $H, H'\in\mathcal{H}_{K'}$ s.t. $|\mathcal{H}_{K'}|=2$, 
then $H$ and $H'$ are adjacent in $\tilde{Q}$
since $\tilde{Q}$ is a support. Therefore, the number of $K'\in\mathcal{K}'$ s.t. $|\mathcal{H}_{K'}|=2$ is at most $3(g+|\mathcal{H}|)$. 

Let $\mathcal{H}'$ be a random subset of $\mathcal{H}$ where each $H\in\mathcal{H}$ is chosen independently with probability
$p=1/\Theta$. Consider a $K'\in\mathcal{K}'$. Since $|\mathcal{H}_{K'}|\le\Theta$, 
the probability that $H, H'\in\mathcal{H}_{K'}$ are the
only subgraphs of $\mathcal{H}_{K'}$ chosen in $\mathcal{H}'$
$p^2(1-p)^{\Theta-2}$, which is at least $1/(e\Theta)^2$. 
The expected number of subgraphs $K'\in\mathcal{K}'$ for which this happens is
$|\mathcal{K}'|/(e\Theta)^2$. 
Since $\E[|\mathcal{H}'|]=|\mathcal{H}|/\Theta$, and there are at most $3(g+|\mathcal{H}|)$ pairs $H, H'\in \mathcal{H}'$ s.t. $\mathcal{H}_{K'}=\{H, H'\}$
for some $K'\in\mathcal{K}'$, by the minimality of $\mathcal{K}'$ it follows that
$|\mathcal{K}'|/(e\Theta)^2\le 3(g+|\mathcal{H}|)/\Theta$. This implies $|\mathcal{K}'|\le 3e^2\Theta\cdot(g+|\mathcal{H}|)=O(\Theta\cdot(g+|\mathcal{H}|))$.
\end{proof}

\crossfreePTAS*

\begin{proof}
Let $\mathcal{L}\subseteq\mathcal{H}$ be a solution returned
by the local-search algorithm
and let $\mathcal{O}\subseteq\mathcal{H}$ be an optimal
solution. Since we are interested in bounding $|\mathcal{L}|$
by $|\mathcal{O}|$, 
we remove the subgraphs in $\mathcal{L}\cap\mathcal{O}$ and replace the capacity
of each $K\in\mathcal{K}$ by its \emph{residual capacity}:
$\cp'(K)=\cp(K)-|\{X:X\cap K\neq\emptyset\mbox{ and }
X\in\mathcal{L}\cap\mathcal{O}\}|$.
We also remove the subgraphs in $\mathcal{K}$ whose residual
capacity is $0$. Let $\mathcal{L}'=\mathcal{L}\setminus(\mathcal{L}\cap\mathcal{O})$ and $\mathcal{O}'=\mathcal{O}\setminus(\mathcal{L}\cap\mathcal{O})$.
If we prove $|\mathcal{L}'|\le(1+\epsilon)|\mathcal{O}'|$, then
it implies that $|\mathcal{L}|\le(1+\epsilon)|\mathcal{O}|$. Hence,
we assume below that $\mathcal{L}\cap\mathcal{O}=\emptyset$, and
each $K\in\mathcal{K}$ has positive capacity bounded above by $\Delta$.

Consider the
intersection system $(G,\mathcal{L}\cup\mathcal{O},\mathcal{K})$. By Lemma~\ref{lem:DeltaK}, there is a collection $\mathcal{K}'\subseteq\mathcal{K}$ 
with $|\mathcal{K}'|=O(\Delta\cdot(g+ |\mathcal{L}\cup\mathcal{O}|))$ s.t.
for any $H, H'\in\mathcal{L}\cup\mathcal{O}$ that belong to $(\mathcal{L}\cup\mathcal{O})_K$ for some $K\in\mathcal{K}$, there is a $K'\in\mathcal{K}'$ s.t.
$H, H'\in(\mathcal{L}\cup\mathcal{O})_{K'}$. 

Now, consider the intersection system $(G,\mathcal{K}',\mathcal{L}\cup\mathcal{O})$. 
By Theorem~\ref{thm:intsupport}, there is an intersection
support $\tilde{Q}$ on $\mathcal{K}'$ 
of genus at most $g$. 
Abusing notation, 
we use $K'$ to also denote the vertex in $\tilde{Q}$ 
corresponding to $K'\in\mathcal{K}'$. 
For each $H\in\mathcal{L}\cup\mathcal{O}$, we use $H$ 
to also denote the induced subgraph of $\tilde{Q}$
defined by the vertices corresponding to subgraphs $K'\in\mathcal{K}'$ s.t. $K'\cap H\neq\emptyset$. 

Let $\Xi=(\mathcal{L}\cup\mathcal{O},E)$ denote the bipartite
intersection graph of $\mathcal{L}$ and $\mathcal{O}$ defined by $\tilde{Q}$.
That is, there is an edge between an $L\in\mathcal{L}$ and
an $O\in\mathcal{O}$ iff there is a $K'\in\mathcal{K}'$
that intersects both $L$ and $O$. In other words, if $L$ and $O$ intersect at a vertex of $\tilde{Q}$.

We claim that $\Xi$ satisfies the local-search condition
for maximization: That is, for any $\mathcal{O}'\subseteq\mathcal{O}$,
the modified solution
$(\mathcal{L}\cup\mathcal{O}')\setminus N(\mathcal{O}')$ is
feasible, where $N(\mathcal{O'})=\cup_{O\in\mathcal{O}'} N_{\Xi}(O)$ is 
the union of neighbors of each $O\in\mathcal{O}'$ in the graph $\Xi$. 
To see this, consider a $K\in\mathcal{K}$. 
If there is an $L\in\mathcal{L}_K$ and
an $O\in\mathcal{O}_K$, then there is a $K'\in\mathcal{K}'$ s.t. $L,O\in(\mathcal{L}\cup\mathcal{O})_{K'}$.
This implies that $L$ and $O$ are adjacent in $\Xi$.
Since $\mathcal{L}$
and $\mathcal{O}$ are feasible, $|\mathcal{L}_K|\le\cp(K)$ and
$|\mathcal{O}_K|\le\cp(K)$. Hence, we added at most $|\mathcal{O}_K|$ subgraphs for $K$, but removed all the subgraphs in $\mathcal{L}_K$. Therefore, 
$(\mathcal{L}\cup\mathcal{O}')\setminus N(\mathcal{O}')$ is feasible.

The graph $\Xi$ may not have bounded genus. Nevertheless,
we show that $\Xi$ has a sub-linear separator. 
Each subgraph $H\in\mathcal{L}\cup\mathcal{O}$ assigns a weight of $1/|H|$ to 
the vertices of $H$ in the intersection support $\tilde{Q}$. 
The weight of a vertex in $\tilde{Q}$ is the sum of the contributions it receives
from the subgraphs containing it. Let $w(K)$ denote the weight of a vertex $K\in\tilde{Q}$
and let $W=|\mathcal{L}\cup\mathcal{O}|$ denote the total weight of the vertices in $\tilde{Q}$. 
Since $\tilde{Q}$ has bounded genus, by Theorem \ref{thm:gilbert},
it has a weighted separator
of size $O(\sqrt{g|\mathcal{K}'|})$. That is, 
a set $S\subseteq V(\tilde{Q})$ of size $O(\sqrt{g|\mathcal{K}'|})$ s.t. $V(\tilde{Q})\setminus S$
separates $\tilde{Q}$ into disconnected components that can be partitioned into two groups
$A$ and $B$ with no edges between them, and that 
$w(A), w(B)\le\alpha W$, for some constant $0<\alpha<1$.

Since $\tilde{Q}$ is an intersection support, each
$L\in\mathcal{L}$ and $O\in\mathcal{O}$ induce a connected
subgraph of $\tilde{Q}$.
Therefore, $S$  yields a separator $S'$ of $\Xi$ consisting
of the subgraphs in $\mathcal{L}\cup\mathcal{O}$ that intersect
the subgraphs corresponding to vertices in $S$. 

Now, we bound $|S'|$.
Consider an arbitrary subgraph $K'\in\mathcal{K}'$.
Since both $\mathcal{L}$ 
and $\mathcal{O}$ are feasible solutions,
$1\le|\mathcal{L}_{K'}|\le\cp(K')$ and $1\le|\mathcal{O}_{K'}|\le\cp(K')$.
Hence, $|\mathcal{L}_{K'}\cup\mathcal{O}_{K'}|\le 2\cp(K')\le2\Delta$.
Therefore, any vertex $K'\in \tilde{Q}$ is contained in 
at most $2\Delta$ subgraphs of $\mathcal{L}\cup\mathcal{O}$.
Hence, $|S'|=O(\Delta\sqrt{g|\mathcal{K'}}|)$. Since $|\mathcal{K}'|=O(\Delta(g+ |\mathcal{L}\cup\mathcal{O}|))$, this implies
that $|S'|=O(\Delta^{3/2}\sqrt{g(g+|\mathcal{L}\cup\mathcal{O}|}))$. 

We finally show that $S'$ is a balanced separator. Since each component of $\tilde{Q}\setminus S$
corresponds to a connected component of $\Xi\setminus S'$, it follows that
$S'$ is a balanced separator since $S$ is a balanced separator.

Since $\Xi$ satisfies both the local-search property and
the sub-linear separator property, it follows
that there is a PTAS via the local-search framework and the framework in~\cite{ChanH12, mustafa2010improved}.
\end{proof}

\packpack*
\begin{proof}
\begin{enumerate}
\item Construct an intersection system $(G,\mathcal{H},\mathcal{K})$, where $\mathcal{K}$ consists of
singleton subgraphs $\{v\}$ for each $v\in V(G)$,
with capacity $\cp(v)$.
The subgraphs in $\mathcal{K}$ are trivially cross-free. Therefore, by Theorem~\ref{thm:crossfreePTAS}, there
is a PTAS for the {\bf Capacitated $\mathcal{H}$-Packing} problem for a cross-free system $(G,\mathcal{H})$ of
genus $g$.
\item Construct an intersection system $(G,\mathcal{V},\mathcal{H})$, where $\mathcal{V}$
consists of singleton subgraphs $\{v\}$ for each $v\in V(G)$, and
each $H\in\mathcal{H}$ has capacity $\cp(H)$. Again, the
intersection system thus defined is cross-free and has genus $g$.
Therefore, by Theorem~\ref{thm:crossfreePTAS}, there is
a PTAS for the {\bf Capacitated Vertex Packing} problem
for $(G,\mathcal{H})$.
\end{enumerate}
\end{proof}

\intcover*
\begin{proof}
Let $\mathcal{L}\subseteq\mathcal{H}$ be a solution returned by local-search and
let $\mathcal{O}\subseteq{\HH}$ be an optimal solution. 
We can assume that $\mathcal{L}\cap\mathcal{O}=\emptyset$.
Otherwise, we can remove the subgraphs in 
$\mathcal{L}\cap\mathcal{O}$, and argue instead
about the subgraphs
in $\mathcal{L}'=\mathcal{L}\setminus(\mathcal{L}\cap\mathcal{O})$
and $\mathcal{O}'=\mathcal{O}\setminus(\mathcal{L}\cap\mathcal{O})$.

Consider the intersection system 
$(G,\mathcal{L}\cup\mathcal{O},\mathcal{K})$.
By Theorem~\ref{thm:intsupport}, there is an intersection
support $\tilde{Q}$ on 
$\mathcal{L}\cup\mathcal{O}$ of genus at most $g$. From $\tilde{Q}$, extract the bipartite
graph $\Xi$ on $\mathcal{L}\cup\mathcal{O}$.
We show that $\Xi$ satisfies the local-search property.

Since $\tilde{Q}$ is a support, for each $K\in\mathcal{K}$, there
is an $L\in\mathcal{L}_K$ adjacent to an $O\in\mathcal{O}_K$ in $\tilde{Q}$,
and hence $L$ and $O$ are adjacent in $\Xi$.
Therefore, for any $\mathcal{L}'\subseteq\mathcal{L}$, $(\mathcal{L}\setminus\mathcal{L}')\cup N(\mathcal{L}')$
is again a feasible solution, where $N(\mathcal{L}')$ refers to the neighborhood of $\mathcal{L}'$ in $\Xi$.
By Theorem \ref{thm:gilbert},
$\tilde{Q}$ has a separator of size $O(\sqrt{g|\mathcal{L}\cup\mathcal{O}|})$, and hence such a separator for $\Xi$.
It follows that $\Xi$ satisfies the local-search Property.
Therefore, there is a PTAS for the Generalized Covering
problem via the local-search framework.
\end{proof}

For the Independent Set and Dominating Set problems, the algorithm follows directly via local-search. Notably, for the
Dominating Set problem, we require an intersection support.
For the Vertex Cover problem, we use the well-known half-integrality due to~\cite{nemhauser1975vertex} of the standard LP-relaxation for Vertex Cover.
That is, there is an optimal solution to the  LP $\min\{\sum x_v: x_u+ x_v\ge 1\;\forall \{u,v\}\in E, x_u\ge 0\}$ where
the $x_u$ take values in $\{0,1/2,1\}$. The vertices whose values are either $1$ or $1/2$ is a feasible vertex cover.
We remove the vertices of weight $0$, take all vertices of weight 1, and in the graph induced on the vertices of weight
$1/2$, it follows that any vertex cover requires at least $1/2$ the vertices. Therefore, a $(1-\epsilon)$-approximation
to the Independent Set problem on this induced subgraph implies a $(1+\epsilon)$-approximation to the Vertex Cover problem.
These ideas have been observed previously by~\cite{bar2011minimum} to obtain a $3/2$-approximation for the Vertex Cover problem
on the intersection graph of axis-aligned rectangles.

\dominatingcrossfree*
\begin{proof}
\begin{enumerate}
\item[$(a)$] The Independent Set problem is equivalent to the Set Packing problem with $\cp(v)=1$ for all $v\in V(G)$ and hence a PTAS follows by Corollary~\ref{cor:packpack}.
\item[$(b)$] For the Vertex Cover problem, we know from the Nemhauser-Trotter theorem as in~\cite{nemhauser1975vertex}
that the standard LP is $1/2$-integral. We only need to find the
subgraphs whose LP-weight is $1/2$. Let $\mathcal{H}'$ be the subgraphs of weight $1/2$. 
For $\mathcal{H}'$ a vertex cover has size at least $|\mathcal{H}'|/2$, since assigning $1/2$ is an optimal solution
for the LP-relaxation on the graph induced on the vertices in $\mathcal{H}'$. Since the complement of a vertex
cover is an independent set, 
using $(a)$ we compute a $(1-\epsilon)$-approximation to the Independent Set problem on $\mathcal{H}'$
and take its complement. This immediately implies a $(1+\epsilon)$-approximation to the Vertex Cover
problem as the solution $VC$ on $\mathcal{H}'$ that is the complement of the $(1-\epsilon)$-approximation to the independent set computed satisfies
$|VC| \le |\mathcal{H}'|-(1-\epsilon)|I|
     \le |\mathcal{H}'|-(1-\epsilon)|\mathcal{H}'|/2$.
This implies that $|VC|\le (1+\epsilon)|\mathcal{H}'|/2\le (1+\epsilon)|VC^*|$.

where $|VC^*|$ denotes the size of an optimal vertex cover for the graph induced on the vertices
of weight $1/2$.
\item[$(c)$]
Consider the cross-free intersection system $(G,\mathcal{H},\mathcal{H})$. 
There is an intersection
support $\tilde{Q}$ of genus at most $g$. 
$\tilde{Q}$ satisfies the local-search conditions and therefore,
the Dominating Set problem on the dual intersection graph $\mathcal{D}$ admits a PTAS
via the local-search framework.
\end{enumerate}
\end{proof}

\subsection{Regions on an Oriented Surface}\label{sec:regionsonsurfaces}
In this section, we give a natural family of geometric hypergraphs that results in cross-free graph systems, and hence the results in the previous section
follow for these hypergraphs.

A simple Jordan curve is the image of an injective map from
$\mathbb{S}^1$ to $\mathbb{R}^2$.  
A set of Jordan curves are a set of \emph{pseudocircles} if 
for any two curves, their boundaries intersect zero or two times.
By the Jordan curve theorem (Ch. 2 in~\cite{Mohar2001GraphsOS}),
a Jordan curve divides the plane into two regions.
For a Jordan curve $\gamma$, we call the
bounded region the \emph{interior} of $\gamma$ and the
unbounded region the \emph{exterior} of $\gamma$.

A set of bounded regions $\mathcal{D}$ is a set of \emph{pseudodisks}
if their boundaries are pseudocircles.
Note that each $D\in\mathcal{D}$ defines a simply
connected region\footnote{A region $R$ is simply connected if any simple loop can be continuously contracted to a point so that
the loop lies in $R$ throughout the contraction process. Intuitively, a simply connected region does not have \emph{holes}.} such that
for any $D,D'\in\mathcal{D}$, $D\setminus D'$ and
$D'\setminus D$ are connected.

If we only insist on the last property, we obtain a set of non-piercing regions, a generalization of pseudodisks.
Recall Definition~\ref{def:npregions} that
a set $\mathcal{D}$ of bounded, connected regions in the plane are a set
of non-piercing regions if for any $D,D'\in\mathcal{D}$,
$D\setminus D'$ and $D'\setminus D$ are connected.
However, we impose the restriction that the boundaries of $D$ and $D'$ intersect a finite number of times for any $D,D'\in\mathcal{D}$. Note that
the sets in $\mathcal{D}$ need not be simply connected.
Fig.~\ref{fig:pseudodisks, piercing and non-piercing} shows examples of pseudodisks,
piercing regions, and non-piercing regions in the plane.
Pseudodisks and non-piercing regions generalize several well-studied geometric objects 
in the plane such as disks, squares, homothets of a convex set, etc. 

\begin{figure}
    \centering
    \includegraphics[width=.85\linewidth]{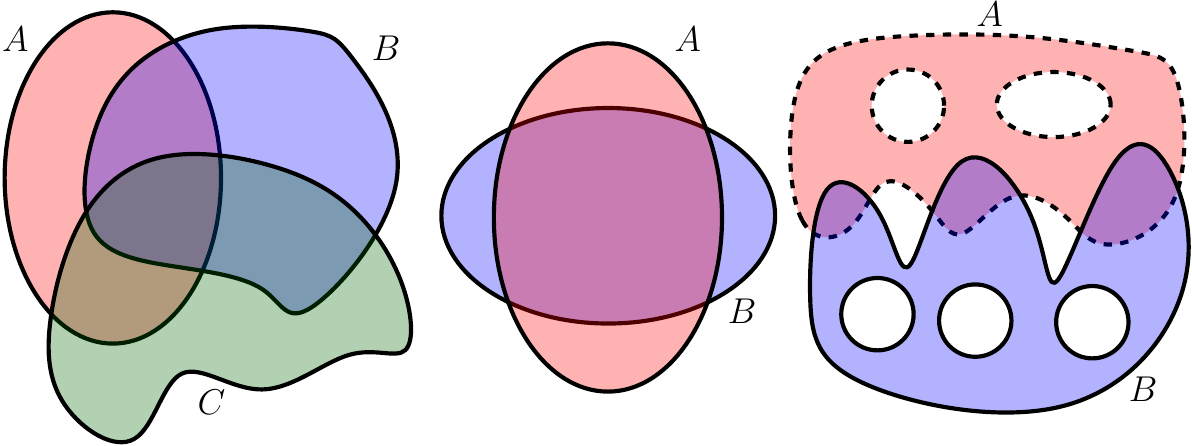}
    \caption{An arrangement of Pseudodisks (left), Piercing (middle), and Non-piercing (right) regions in the plane.}
    \label{fig:pseudodisks, piercing and non-piercing}
\end{figure}

Next, we introduce the intersection hypergraph of regions on a surface. 
The notion is similar to the graph-theoretic notion, but we define it in terms of regions for completeness.

\begin{definition}[Intersection hypergraph of regions] Let $\HH$ and $\KK$ be two collections of regions on an surface.
The intersection hypergraph of the regions is a hypergraph $(\HH,\{\HH_K\}_{K\in\KK})$, i.e., it takes a vertex for each $H\in\HH$, and a hyperedge $\HH_K$ for each $K\in\KK$ where $\HH_K=\{H\in\HH:H\cap K\ne\emptyset\}$.    
\end{definition}

Generalizing previous works~\cite{BasuRoy2018, ChanH12, mustafa2010improved}, the authors in~\cite{RR18} showed that the intersection
hypergraph of non-piercing regions in the plane admits a planar support. 
The results of~\cite{RR18} however, require that the non-piercing regions are in \emph{general position}. That is,
the boundaries of any pair of regions 
intersect only at a finite number of points where they cross, and the boundaries of no three regions intersect at a common point.
In addition, they require that
the points in the input are at some positive distance away from the boundary of the regions. In particular,
the non-piercing assumption does not include regions that share a part of their boundary, or have points on the
boundary of the regions, such as the regions shown in Fig.~\ref{fig:touching}. 
To handle regions of this type, we introduce the notion of \emph{weakly non-piercing regions} defined as follows.

\begin{definition}[Weakly non-piercing regions]\label{def:weaklynp}
A set $\mathcal{R}$ of connected regions on an oriented surface is said to be weakly non-piercing if for each pair of regions $R,R'\in\mathcal{R}$, either $R\setminus R'$ or $R'\setminus R$ is connected.
\end{definition}

In the Conclusion section of~\cite{RR18}, the authors show an example of a set of weakly non-piercing regions in the
plane that does not admit a planar support. 
We show that if we additionally impose the condition 
that the regions are
\emph{simply connected}, 
then there is a support of genus at most $g$.

\begin{figure}[h!]
\begin{center}
    \includegraphics[width=2in]{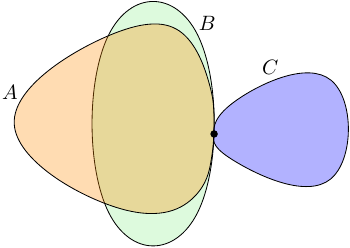}
    \caption{The regions $A$ colored orange, $B$ colored green, and $C$ colored blue are touching. Observe that
$A\setminus B$ is connected, while $B\setminus A$ induces two connected components. There is an input point on the boundary
of the three regions. The regions are weakly non-piercing, but are not non-piercing.}
\label{fig:touching}
\end{center}
\end{figure}

Let $\mathcal{R}$ be a collection of weakly
non-piercing regions on a surface of genus $g$.
Consider the boundaries $\partial R$ and $\partial R'$ of two regions
$R$ and $R'$ in $\mathcal{R}$, respectively. Then, $\partial R\cap\partial R'$ consists of a finite collection of
isolated points and arcs. If at an isolated point, the boundaries of the regions cross, then we call this point
a \emph{crossing point}. Otherwise, we call it a \emph{touching point}. Equivalently, a point $p\in\partial R\cap \partial R'$ is a touching point if in an arbitrarily small ball $B$ around $p$, the boundaries $\partial R$ and
$\partial R'$ can be modified by an arbitrarily small amount so that $\partial R$ and $\partial R'$ are disjoint in $B$.

For an intersection hypergraph defined by weakly non-piercing regions
$\HH$ and $\KK$ on a surface $\Sigma$, we now show that we can construct a graph
$G$ embedded on $\Sigma$, and subgraphs for each $H\in\HH$ and $K\in\KK$ s.t. they induce a cross-free intersection graph system.

\touchcross*
\begin{proof}
Let $\mathcal{R}=\HH\cup\KK$. The arrangement of the regions of $\HH\cup\KK$
partition the surface into \emph{vertices}, \emph{edges} and \emph{cells}.
There are two kinds of vertices in the arrangement: when the boundaries of
two regions cross, the intersection point is called a \emph{crossing vertex}.
For each region $R\in\mathcal{R}$ we subdivide its boundary into 
arcs such that for each arc $a$ on the boundary of $R$, and for any two points $p,q$ in the interior
of $a$, the boundaries of set of regions of $\mathcal{R}$ containing $p$ on its boundary, 
and the set of regions in $\mathcal{R}$ containing
$q$ on its boundary are identical. We
add the end-points of all such arcs $a$ to the set of vertices, and we call
these the \emph{touching vertices}. The arc between two consecutive
vertices on the boundary of a region are the edges of the arrangement. For 
regions in $\mathcal{R}$ not intersected by other regions, the edge contributed
by this region is its boundary.
Finally, the connected regions obtained by removing the vertices and edges
in the arrangement are the cells of the arrangement.

The vertices and edges of the arrangement define a graph $G'$. To obtain
a host graph $G$, we further subdivide each edge of $G'$ by a vertex. 
Next, we put a vertex
in the interior of each cell in the arrangement. We join the vertex corresponding 
to a cell to the vertices on its boundary
via non-crossing simple arcs that lie in the cell. Since each cell is connected, we can do so.
The graph $G$ is clearly embedded on the surface.
See Fig.~\ref{fig:graphofregions} for an example.

Consider a region $R\in\mathcal{R}$. Since $R$ is connected, by construction, it defines
a connected subgraph of $G$. Further, since $R$ is simply connected, its boundary
is a Jordan curve, and by construction is mapped to a simple cycle $C_R$ in $G$ s.t. 
this cycle separates the vertices of $G$ that lie in the subgraph corresponding to $R$
from the vertices of $G$ that do not.

We now show that the regions in $\HH$ induce cross-free subgraphs of $G$.
The fact that the regions in $\KK$ induce cross-free subgraphs of $G$ will then follow analogously. 
Abusing notation, we use $H$ to also denote the subgraph of $G$ 
induced on the vertices in $H$.
For contradiction,
suppose two subgraphs $H,H'\in G$ corresponding respectively, to regions
$H, H'\in\mathcal{H}$ are crossing. 
Then the reduced graph (See Definition~\ref{def:reducedgraph})
$R_G(H, H')$ contains a vertex $v$ and four edges in cyclic order, to vertices $v_1, v_2, v_3, v_4$ s.t. $v_1, v_3\in H\setminus H'$ and $v_2, v_4\in H'\setminus H$.
Since the corresponding regions $H, H'$ are weakly non-piercing, either
$H\setminus H'$ or $H'\setminus H$ is connected. Suppose wlog, the former holds.
Then, in $G$, there is a path $P$ from $v_1$ to $v_3$ that lies entirely in
the subgraph $H\setminus H'$. Adding the edges $\{v_1, v\}$ and
$\{v, v_3\}$ to $P$, we obtain a simple cycle $C$, all of whose vertices lie
in $H$. Since the boundary of the region $H$ defines a simple cycle $C_H$ in $G$ separating
the vertices in $H$ from those that do not lie in $H$, it implies that
$C$ lies in the subgraph defined by $H$.
But this implies that either $v_2$ or $v_4$ also lies in $H$, contradicting
the assumption that $H$ and $H'$ are crossing.
Therefore, the subgraphs defined by regions
in $\mathcal{H}$ induce a cross-free collection of subgraphs in $G$, and
analogously, so do the subgraphs defined by the regions in $\mathcal{K}$. Hence, $(G,\HH',\KK')$ is
a cross-free intersection system, where $\HH'$ and $\KK'$ are respectively,
the subgraphs in $G$ corresponding to the regions in $\HH$ and $\KK$.

Moreover, by construction, any two regions $R_1,R_2\in\HH\cup\KK$ intersect if and only if the corresponding subgraphs $R'_1,R'_2\in\HH'\cup\KK'$ share a vertex in $G$.
This implies the intersection hypergraph defined by $\HH$ and $\KK$, and that defined by $(G,\mathcal{H}',\mathcal{K}')$ are isomorphic.
\end{proof}

\begin{figure}[h!]
\begin{center}
    \includegraphics[width=2in]{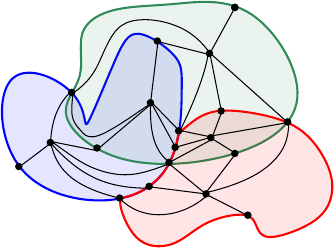}
    \caption{Construction of cross-free graph system for simply connected regions.}
\label{fig:graphofregions}
\end{center}
\end{figure}

By Theorem \ref{thm:touchcross}, all the results for packing and covering problems discussed in Subsection \ref{sec:packcover}, also hold for hypergraphs defined by simply connected weakly non-piercing regions on a surface of bounded genus.

\subsubsection{Hardness}
We believe that the cross-free condition is essential to obtain PTAS for packing
and covering problems when the host graph has bounded genus.~\cite{DBLP:journals/comgeo/ChanG14}
proved that for a hypergraph defined by a set of horizontal and vertical slabs 
in the plane and a set of points $P$, the Hitting Set problem and
the Set Cover problems are  APX-hard. A simple modification of their result implies the following.

\apxhard*
\begin{proof}
The proof follows directly from the corresponding APX-hardness proof of~\cite{DBLP:journals/comgeo/ChanG14}.
We only sketch the modification required.
Consider the Set Cover problem: Given a set of horizontal slabs $H$, a set $V$ of vertical slabs, and a set $P$ of points in the
plane. The authors show that it is APX-hard to select a minimum cardinality subset of $H\cup V$ to cover $P$. 
To obtain the claimed APX-hardness proof on the torus for non-piercing regions, we embed this construction on a torus, and then
modify the boundary of each region in $H$ to be a pair of parallel non-separating closed curves parallel to the hole.
Similarly, we map each vertical slab in $V$ to a region bounded by two parallel non-separating closed curves perpendicular to the hole.

Now, construct the dual arrangement graph $G$ with a representative point for each non-empty cell in the arrangement of the regions,
and let $\mathcal{H}$ denote the set of subgraphs of $G$ defined by the regions $H\cup V$.
In $(G,\mathcal{H})$, the subgraphs are non-piercing,
but are crossing. The APX-hardness of the problem follows from the corresponding result of~\cite{DBLP:journals/comgeo/ChanG14}.

The proof of APX-hardness for the Hitting Set problem for non-piercing crossing subgraphs of a graph follows by a similar modification
of the construction in~\cite{DBLP:journals/comgeo/ChanG14} for the Hitting Set problem with horizontal and vertical slabs in the plane.
\end{proof}

\subsection{Coloring Geometric Hypergraphs}\label{sec:color_geom}
~\cite{KellerS18} showed that the intersection hypergraph
of disks in the plane is 4-colorable i.e., a coloring with 4 colors such that no hyperedge is monochromatic (assuming that the size of each hyperedge is at least two).
This was generalized by~\cite{Keszegh20} for pseudodisks, which was 
further generalized by~\cite{RR18} to show that the intersection hypergraph of non-piercing regions is 4-colorable.
To bound the chromatic number of the hypergraph, suppose we had a 
weak support $G$ corresponding to the vertices of the hypergraph 
whose chromatic number is bounded by $k$. Since a $k$-coloring of $G$
is also a feasible coloring of the hypergraph, it implies that
the hypergraph can also be $k$-colored.
Since any support is also a weak support, 
as a consequence of Theorem~\ref{thm:intsupport} and Theorem~\ref{thm:touchcross},
we obtain the following.

\colorhypergraph*
\begin{proof}
By Theorem \ref{thm:touchcross}, there is a cross-free intersection system $(G,\HH',\KK')$ such that the intersection hypergraph defined by $\HH$ and $\KK$ is isomorphic to that defined by $\HH'$ and $\KK'$.
By Theorem~\ref{thm:intsupport}, $(G,\mathcal{H}',\mathcal{K}')$ has an intersection support $\tilde{Q}$ of genus at most $g$.
Now, $\chi(\tilde{Q})\le \frac{7 + \sqrt{1+24g}}{2}$ (see~\cite{diestel2005graph} for a reference). Since $\tilde{Q}$ is a support,
for each $K\in\mathcal{K'}$, there is an edge between some two subgraphs $H,H'\in {\mathcal{H}'_K}$.
Therefore,
no hyperedge ${\mathcal{H}'_K}$ is monochromatic.
\end{proof}

\section{Conclusion}
\label{sec:conclusion}
In this paper, we studied the problem of constructing primal, dual supports for graph systems $(G,\mathcal{H})$ defined on a host graph
$G$. We also considered the more general problem of constructing a support for an intersection system $(G,\mathcal{H},\mathcal{K})$.
We showed that
if $G$ has bounded genus, then the cross-free property is sufficient to obtain a support of genus at most that of $G$.

There are several intriguing open questions and research directions,
and we mention a few: We do not know if the algorithms to construct a primal, dual or intersection support run in polynomial time.

We obtained PTAS for the Generalized Packing Problems with bounded capacities. 
For Generalized Covering Problem with bounded demands however, we could not obtain a PTAS.~\cite{DBLP:journals/dcg/RamanR22} obtained a PTAS for covering points in the plane whose demands are bounded by a constant, with
non-piercing regions. Combining the results of~\cite{DBLP:journals/dcg/RamanR22} with the Sweeping theorem of~\cite{DBLP:conf/compgeom/DalalG0R24}, we can obtain a PTAS for hitting non-piercing regions in the plane with bounded demands by points from a set $P$.
In the graph setting however, there is no equivalent 
Sweeping Theorem. Developing such a Sweeping Theorem is an interesting direction for future research.

The construction of a support required an embedding of the graph that is cross-free with respect
to the subgraphs. What is the complexity of deciding if a graph of bounded genus has a cross-free embedding with respect to a collection
of subgraphs?
A broader line of research is to obtain necessary and sufficient conditions for a hypergraph to have a \emph{sparse support} - 
where sparsity could be a graph with sublinear-sized separators or even just a graph with a linear number of edges.

\acknowledgements
\label{sec:ack}
Part of this work was carried out when the first author was at LIMOS, Université Clermont Auvergne, and was partially supported by the French government research program “Investissements d’Avenir” through the IDEX-ISITE initiative 16-IDEX-0001 (CAP 20-25).

\bibliographystyle{abbrvnat}
\bibliography{ref}

@Inbook{Feinberg1997,
author={Feinberg, V.
and Levin, A.
and Rabinovich, E.},
title={Hypergraph Planarization},
bookTitle={VLSI Planarization: Methods, Models, Implementation},
year={1997},
publisher={Springer Netherlands},
address={Dordrecht},
pages={45-86},
isbn={978-94-011-5740-7},
doi={10.1007/978-94-011-5740-7_4}
}

@article{even_conflictfree,
author = {Even, Guy and Lotker, Zvi and Ron, Dana and Smorodinsky, Shakhar},
title = {Conflict-Free Colorings of Simple Geometric Regions with Applications to Frequency Assignment in Cellular Networks},
journal = {SIAM Journal on Computing},
volume = {33},
number = {1},
pages = {94-136},
year = {2003},
doi = {10.1137/S0097539702431840}
}

@article{gilbert1984separator,
  title={A separator theorem for graphs of bounded genus},
  author={Gilbert, John R and Hutchinson, Joan P and Tarjan, Robert Endre},
  journal={Journal of Algorithms},
  volume={5},
  number={3},
  pages={391--407},
  year={1984},
  publisher={Elsevier},
  doi={10.1016/0196-6774(84)90019-1}
}

@article{buchin2011planar,
  title={On planar supports for hypergraphs},
  author={Buchin, Kevin and van Kreveld, Marc J and Meijer, Henk and Speckmann, Bettina and Verbeek, KAB},
  journal={Journal of Graph Algorithms and Applications},
  volume={15},
  number={4},
  pages={533--549},
  year={2011},
  doi = {10.7155/jgaa.00237},
  publisher={Brown University}
}

@inproceedings{DBLP:conf/compgeom/DalalG0R24,
  author       = {Suryendu Dalal and
                  Rahul Gangopadhyay and
                  Rajiv Raman and
                  Saurabh Ray},
  title        = {Sweeping Arrangements of Non-Piercing Regions in the Plane},
  booktitle    = {40th International Symposium on Computational Geometry, (SoCG 2024),
                  June 11-14, 2024, Athens, Greece},
  series       = {Leibniz International Proceedings in Informatics (LIPIcs)},
  volume       = {293},
  pages        = {45:1--45:15},
  publisher    = {Schloss Dagstuhl - Leibniz-Zentrum f{\"{u}}r Informatik},
  year         = {2024},
  doi          = {10.4230/LIPICS.SOCG.2024.45}
}

@article{voloshina1984planarity,
  title={O planarnosti gipergrafov [On the Planarity of Hypergraphs]},
  author={Voloshina, A.A. and Feinberg, V.Z.},
  journal = {Doklady of the Academy of Sciences of Belarus},
  volume={28},
  pages={309--311},
  year={1984},
  organization={Academy of Sciences of Belarus}
  }

@book{feinberg2012vlsi,
  title={VLSI planarization: methods, models, implementation},
  author={Feinberg, VZ and Levin, AG and Rabinovich, EB},
  volume={399},
  year={1997},
  publisher={Springer Science \& Business Media},
  series={ Mathematics and Its Applications},
  doi={10.1007/978-94-011-5740-7}
}

@inproceedings{brandes2010blocks,
  title={Blocks of hypergraphs: applied to hypergraphs and outerplanarity},
  author={Brandes, Ulrik and Cornelsen, Sabine and Pampel, Barbara and Sallaberry, Arnaud},
  booktitle={Combinatorial Algorithms: 21st International Workshop on Combinatorial Algorithms (IWOCA 2010), July 26-28, 2010, London, UK},
  pages={201--211},
  series = {Lecture Notes in Computer Science (LNCS)},
  volume = {6460},
  year={2011},
  organization={Springer},
  doi = {10.1007/978-3-642-19222-7_21}
  
}

@inproceedings{bereg2015colored,
  title={Colored non-crossing {E}uclidean {S}teiner forest},
  author={Bereg, Sergey and Fleszar, Krzysztof and Kindermann, Philipp and Pupyrev, Sergey and Spoerhase, Joachim and Wolff, Alexander},
  booktitle={26th International Symposium on Algorithms and Computation (ISAAC 2015), December 9-11, 2015, Nagoya, Japan},
  series = {Lecture Notes in Computer Science (LNCS)},
  volume = {9472},
  pages={429--441},
  year={2015},
  organization={Springer},
  doi = {10.1007/978-3-662-48971-0_37}
}

@article{bereg2011red,
  title={On the red/blue spanning tree problem},
  author={Bereg, Sergey and Jiang, Minghui and Yang, Boting and Zhu, Binhai},
  journal={Theoretical computer science},
  volume={412},
  number={23},
  pages={2459--2467},
  year={2011},
  publisher={Elsevier},
  doi = {10.1016/j.tcs.2010.10.038}
}

@article{havet2022overlaying,
  title={Overlaying a hypergraph with a graph with bounded maximum degree},
  author={Havet, Fr{\'e}d{\'e}ric and Mazauric, Dorian and Nguyen, Viet-Ha and Watrigant, R{\'e}mi},
  journal={Discrete Applied Mathematics},
  volume={319},
  pages={394--406},
  year={2022},
  publisher={Elsevier},
  doi = {10.1016/j.dam.2022.05.022}
}

@article{tarjan1984simple,
  title={Simple linear-time algorithms to test chordality of graphs, test acyclicity of hypergraphs, and selectively reduce acyclic hypergraphs},
  author={Tarjan, Robert E and Yannakakis, Mihalis},
  journal={SIAM Journal on computing},
  volume={13},
  number={3},
  pages={566--579},
  year={1984},
  publisher={SIAM},
  doi={10.1137/0213035}
}

@article{korach2003clustering,
  title={The clustering matroid and the optimal clustering tree},
  author={Korach, Ephraim and Stern, Michal},
  journal={Mathematical Programming},
  volume={98},
  pages={385--414},
  year={2003},
  publisher={Springer},
  doi={10.1007/s10107-003-0410-x}
}

@article{alon1990separator,
  title={A separator theorem for nonplanar graphs},
  author={Alon, Noga and Seymour, Paul and Thomas, Robin},
  journal={Journal of the American Mathematical Society},
  volume={3},
  number={4},
  pages={801--808},
  year={1990},
  doi={10.2307/1990903}
}

@article{RR18,
  title={Constructing planar support for non-piercing regions},
  author={Raman, Rajiv and Ray, Saurabh},
  journal={Discrete \& Computational Geometry},
  volume={64},
  number={3},
  pages={1098--1122},
  year={2020},
  publisher={Springer},
  doi={10.1007/s00454-020-00216-w}
}

@article{krohn2014guarding,
  title={Guarding terrains via local search},
  author={Krohn, Erik and Gibson, Matt and Kanade, Gaurav and Varadarajan, Kasturi},
  journal={Journal of Computational Geometry},
  volume={5},
  number={1},
  pages={168--178},
  year={2014},
  doi={10.20382/jocg.v5i1a9}
}

@inproceedings{DBLP:conf/wads/BandyapadhyayR17,
  author       = {Sayan Bandyapadhyay and
                  Aniket Basu Roy},
  editor       = {Faith Ellen and
                  Antonina Kolokolova and
                  J{\"{o}}rg{-}R{\"{u}}diger Sack},
  title        = {Effectiveness of Local Search for Art Gallery Problems},
  booktitle    = {Algorithms and Data Structures - 15th International Symposium, {WADS}
                  2017, St. John's, NL, Canada, July 31 - August 2, 2017, Proceedings},
  series       = {Lecture Notes in Computer Science},
  volume       = {10389},
  pages        = {49--60},
  publisher    = {Springer},
  year         = {2017},
  doi          = {10.1007/978-3-319-62127-2\_5},
  bibsource    = {dblp computer science bibliography, https://dblp.org}
}

@article{mustafa2015quasi,
  title={Quasi-Polynomial Time Approximation Scheme for Weighted Geometric Set Cover on Pseudodisks and Halfspaces},
  author={Mustafa, Nabil H. and Raman, Rajiv and Ray, Saurabh},
  journal={SIAM Journal on Computing},
  volume={44},
  number={6},
  pages={1650--1669},
  year={2015},
  publisher={SIAM},
  doi={10.1137/14099317X}
}

@Article{BasuRoy2018,
author={Basu Roy, Aniket
and Govindarajan, Sathish
and Raman, Rajiv
and Ray, Saurabh},
title={Packing and Covering with Non-Piercing Regions},
journal={Discrete {\&} Computational Geometry},
year={2018},
volume = {60},
pages = {471--492},
doi = {10.1007/s00454-018-9983-2}
}

@article{clarkson1989applications,
  title={Applications of random sampling in computational geometry, {II}},
  author={Clarkson, Kenneth L and Shor, Peter W},
  journal={Discrete \& Computational Geometry},
  volume={4},
  number={5},
  pages={387--421},
  year={1989},
  publisher={Springer},
  doi={10.1007/BF02187740}
}

@inproceedings{KellerS18,
  author    = {Chaya Keller and
               Shakhar Smorodinsky},
  title     = {Conflict-Free Coloring of Intersection Graphs of Geometric Objects},
  booktitle = {Proceedings of the Twenty-Ninth Annual {ACM-SIAM} Symposium on Discrete
               Algorithms, {SODA} 2018, New Orleans, LA, USA, January 7-10, 2018},
  pages     = {2397--2411},
  year      = {2018},
  doi={10.5555/3174304.3175459}
}

@article{CabelloG14,
  title={Simple {PTAS}'s for families of graphs excluding a minor},
  author={Sergio Cabello and David Gajser},
  journal={Discret. Appl. Math.},
  year={2015},
  volume={189},
  pages={41-48},
  doi={10.1016/j.dam.2015.03.004}
}

@article{sharir2003clarkson,
  title={The Clarkson--Shor technique revisited and extended},
  author={Sharir, Micha},
  journal={Combinatorics, Probability and Computing},
  volume={12},
  number={2},
  pages={191--201},
  year={2003},
  publisher={Cambridge University Press},
  doi={10.1017/S0963548302005527}
}

@article{ChanH12,
  author    = {Timothy M. Chan and
               Sariel Har{-}Peled},
  title     = {Approximation Algorithms for Maximum Independent Set of Pseudo-Disks},
  journal   = {Discrete {\&} Computational Geometry},
  volume    = {48},
  number    = {2},
  pages     = {373--392},
  year      = {2012},
  doi       = {10.1007/s00454-012-9417-5},
  bibsource = {dblp computer science bibliography, https://dblp.org}
}

@article{onus2011minimum,
  title={Minimum maximum-degree publish--subscribe overlay network design},
  author={Onus, Melih and Richa, Andr{\'e}a W},
  journal={IEEE/ACM Transactions on Networking},
  volume={19},
  number={5},
  pages={1331--1343},
  year={2011},
  publisher={IEEE},
  doi={10.1109/TNET.2011.2144999}
}

@article{hosoda2012approximability,
  title={On the approximability and hardness of minimum topic connected overlay and its special instances},
  author={Hosoda, Jun and Hromkovi{\v{c}}, Juraj and Izumi, Taisuke and Ono, Hirotaka and Steinov{\'a}, Monika and Wada, Koichi},
  journal={Theoretical Computer Science},
  volume={429},
  pages={144--154},
  year={2012},
  publisher={Elsevier},
  doi = {10.1016/j.tcs.2011.12.033}
}

@article{brandes2012path,
  title={Path-based supports for hypergraphs},
  author={Brandes, Ulrik and Cornelsen, Sabine and Pampel, Barbara and Sallaberry, Arnaud},
  journal={Journal of Discrete Algorithms},
  volume={14},
  pages={248--261},
  year={2012},
  publisher={Elsevier},
  doi = {10.1016/j.jda.2011.12.009}
}

@article{Johnson1987Hypergraph,
  title={Hypergraph planarity and the complexity of drawing Venn diagrams},
  author={Johnson, David S and Pollak, Henry O},
  journal={Journal of graph theory},
  volume={11},
  number={3},
  pages={309--325},
  year={1987},
  publisher={Wiley Online Library},
  doi={10.1002/jgt.3190110306}
}

@article{chen2015polynomial,
  title={Polynomial-time data reduction for the subset interconnection design problem},
  author={Chen, Jiehua and Komusiewicz, Christian and Niedermeier, Rolf and Sorge, Manuel and Suchy, Ondrej and Weller, Mathias},
  journal={SIAM Journal on Discrete Mathematics},
  volume={29},
  number={1},
  pages={1--25},
  year={2015},
  publisher={SIAM},
  doi={10.1137/140955057}
}

@article{zykov,
  title={Hypergraphs},
  author={Zykov, A.A.},
  journal={Russian Mathematical Surveys},
  volume={29},
  number={6},
  pages={89--156},
  year={1974},
  publisher={IOP Publishing},
  doi={10.1070/RM1974v029n06ABEH001303}
}

@article{mustafa2010improved,
  title={Improved results on geometric hitting set problems},
  author={Mustafa, Nabil H and Ray, Saurabh},
  journal={Discrete \& Computational Geometry},
  volume={44},
  number={4},
  pages={883--895},
  year={2010},
  publisher={Springer},
  doi={10.1007/s00454-010-9285-9}
}

@article{Keszegh20,
  author    = {Bal{\'{a}}zs Keszegh},
  title     = {Coloring Intersection Hypergraphs of Pseudo-Disks},
  journal   = {Discret. Comput. Geom.},
  volume    = {64},
  number    = {3},
  pages     = {942--964},
  year      = {2020},
  doi       = {10.1007/s00454-019-00142-6},
  bibsource = {dblp computer science bibliography, https://dblp.org}
}

@article{weisstein2016torus,
  title={Torus grid graph},
  author={Weisstein, Eric W},
  year={2016},
  publisher={Wolfram Research, Inc.},
  url={https://mathworld.wolfram.com/TorusGridGraph.html}
}

@article{DBLP:journals/comgeo/ChanG14,
  author    = {Timothy M. Chan and
               Elyot Grant},
  title     = {Exact algorithms and {APX}-hardness results for geometric packing and
               covering problems},
  journal   = {Comput. Geom.},
  volume    = {47},
  number    = {2},
  pages     = {112--124},
  year      = {2014},
  doi       = {10.1016/j.comgeo.2012.04.001},
  bibsource = {dblp computer science bibliography, https://dblp.org}
}

@inproceedings{DBLP:conf/walcom/AschnerKMY13,
  author    = {Rom Aschner and
               Matthew J. Katz and
               Gila Morgenstern and
               Yelena Yuditsky},
  title     = {Approximation Schemes for Covering and Packing},
  booktitle = {7th International Workshop on Algorithms and Computation,
               ({WALCOM} 2013), February 14-16, 2013, Kharagpur, India},
  series    = {Lecture Notes in Computer Science (LNCS)},
  volume    = {7748},
  pages     = {89--100},
  publisher = {Springer},
  year      = {2013},
  doi       = {10.1007/978-3-642-36065-7\_10},
}

@book{diestel2005graph,
  title={Graph Theory},
  author={Diestel, Reinhard},
  series={Graduate texts in mathematics},
  volume={173},
  number={33},
  publisher = {Springer},
  year={2005},
  doi = {10.1007/978-3-662-70107-2},
}

@article{hurtado2018colored,
  title={Colored spanning graphs for set visualization},
  author={Hurtado, Ferran and Korman, Matias and van Kreveld, Marc and L{\"o}ffler, Maarten and Sacrist{\'a}n, Vera and Shioura, Akiyoshi and Silveira, Rodrigo I and Speckmann, Bettina and Tokuyama, Takeshi},
  journal={Computational Geometry},
  volume={68},
  pages={262--276},
  year={2018},
  publisher={Elsevier},
  doi={10.1016/j.comgeo.2017.06.006}
}

@inproceedings{baldoni2007tera,
  title={{TERA}: topic-based event routing for peer-to-peer architectures},
  author={Baldoni, Roberto and Beraldi, Roberto and Quema, Vivien and Querzoni, Leonardo and Tucci-Piergiovanni, Sara},
  booktitle={Inaugural International Conference on Distributed Event-Based Systems, (DEBS'07), June 20-22, 2007, Toronto, Ontario, Canada},
  publisher = {Association for Computing Machinery},
  doi = {10.1145/1266894.1266898},
  series = {Distributed Event-Based Systems (DEBS)},
  pages={2--13},
  year={2007}
}

@article{DBLP:journals/dcg/RamanR22,
  author    = {Rajiv Raman and
               Saurabh Ray},
  title     = {On the Geometric Set Multicover Problem},
  journal   = {Discret. Comput. Geom.},
  volume    = {68},
  number    = {2},
  pages     = {566--591},
  year      = {2022},
  doi       = {10.1007/s00454-022-00402-y},
  bibsource = {dblp computer science bibliography, https://dblp.org}
}

@article{ackerman2020coloring,
  title={Coloring hypergraphs defined by stabbed pseudo-disks and {ABAB}-free hypergraphs},
  author={Ackerman, Eyal and Keszegh, Bal{\'a}zs and P{\'a}lv{\"o}lgyi, D{\"o}m{\"o}t{\"o}r},
  journal={SIAM Journal on Discrete Mathematics},
  volume={34},
  number={4},
  pages={2250--2269},
  year={2020},
  publisher={SIAM},
  doi          = {10.1137/19M1290231}
}

@book{Mohar2001GraphsOS,
  title={Graphs on surfaces},
  author={Bojan Mohar and Carsten Thomassen},
  series={Johns Hopkins series in the mathematical sciences},
  year={2001},
  publisher = {Johns Hopkins University Press},
  doi={10.56021/9780801866890}
}

@article{hastad1999clique,
  title={Clique is hard to approximate within $n^{1-\varepsilon}$},
  author={H\r{a}stad, J},
  journal={Acta Mathematica},
  volume={182},
  number={1},
  pages={105--142},
  year={1999},
  doi={10.1007/BF02392825}
}

@article{chvatal1979greedy,
  title={A greedy heuristic for the set-covering problem},
  author={Chvatal, Vasek},
  journal={Mathematics of operations research},
  volume={4},
  number={3},
  pages={233--235},
  year={1979},
  publisher={INFORMS},
  doi={10.1287/moor.4.3.233}
}

@article{lov1975notdef,
  title={On the ratio of optimal integral and fractional covers},
  author={Lov\'{a}sz, Laszlo},
  journal={Discrete Math},
  volume={13},
  pages={383--390},
  year={1975},
  doi={10.1016/0012-365X(75)90058-8}
}

@article{feige1998threshold,
  title={A threshold of ln n for approximating set cover},
  author={Feige, Uriel},
  journal={Journal of the ACM (JACM)},
  volume={45},
  number={4},
  pages={634--652},
  year={1998},
  publisher={ACM New York, NY, USA},
  doi={10.1145/285055.285059}
  }

@article{nemhauser1975vertex,
  title={Vertex packings: structural properties and algorithms},
  author={Nemhauser, George L and Trotter Jr, Leslie E},
  journal={Mathematical Programming},
  volume={8},
  number={1},
  pages={232--248},
  year={1975},
  publisher={Springer},
  doi={10.1007/BF01580444}
}

@article{bar2011minimum,
  title={Minimum vertex cover in rectangle graphs},
  author={Bar-Yehuda, Reuven and Hermelin, Danny and Rawitz, Dror},
  journal={Computational Geometry},
  volume={44},
  number={6-7},
  pages={356--364},
  year={2011},
  publisher={Elsevier},
  doi={10.1016/j.comgeo.2011.03.002}
}

@techreport{whitesides1990k,
  author      = {Sue Whitesides and Ruidong Zhao},
  title       = {K-admissible collections of Jordan curves and offsets of circular arc figures},
  institution = {McGill University},
  number      = {SOCS-90.08},
  address     = {Montreal},
  year        = {1990},
  url={https://scholar.google.com/scholar?q=K-admissible+collections+of+Jordan+curves+and+offsets+of+circular+arc+figures}
}

@inproceedings{raman2023hypergraph,
    title={On Hypergraph Supports}, 
      author={Rajiv Raman and Karamjeet Singh},
    booktitle = {Proceedings of the 12th European Conference on Combinatorics, Graph Theory and Applications (EUROCOMB’23)},
    year = {2023},
    doi = {10.5817/CZ.MUNI.EUROCOMB23-107},
}
\label{sec:biblio}

\end{document}